\documentclass{article}
\usepackage{mathrsfs}%
\usepackage[title]{appendix}%
\usepackage{xcolor}%
\usepackage{textcomp}%
\usepackage{manyfoot}%
\usepackage{booktabs}%
\usepackage{algorithm}%
\usepackage{algorithmic}
\usepackage{listings}%
\usepackage{amsthm}%
\usepackage{enumitem}%
\usepackage{amssymb}%
\usepackage{amsmath}%
\usepackage{float}%
\usepackage{graphicx}%
\usepackage{xurl}

\newcommand{\nobracket}{}

\newcommand{\tmname}[1]{\textsc{#1}}
\newcommand{\tmop}[1]{#1}          
\newcommand{\tmstrong}[1]{\textbf{#1}}

\newcommand{\Fq}{\mathbb{F}_q}
\newcommand{\Fqm}{\mathbb{F}_{q^m}}
\newcommand{\QP}{\mathbb{F}_{q^m} \langle X^q \rangle}

\newcommand{\SM}[2]{\text{SM}(#1,#2)}
\newcommand{\sO}{\tilde{\mathcal{O}}}
\newcommand{\GF}[1]{\mathbb{F}_{#1}}          
\newcommand{\supp}{\operatorname{supp}}
\newcommand{\LC}{\operatorname{LC}}

\theoremstyle{definition}
\newtheorem{theorem}{Theorem}
\newtheorem{proposition}[theorem]{Proposition}
\newtheorem{lemma}[theorem]{Lemma}

\newtheorem{remark}{Remark}
\newtheorem{definition}{Definition}
\newtheorem{corollary}{Corollary}

\author{
      Philippe Gaborit\thanks{XLIM, France, MATHIS/CRYPTIS, \texttt{philippe.gaborit@unilim.fr}} \\
      \and
      Camille Garnier\thanks{XLIM, France, MATHIS/CRYPTIS, \texttt{camille.garnier@unilim.fr}} \thanks{Research financially supported by the Military French Ministry – Defense and Innovation Agency (DGA-AID).}\\
      \and
      Olivier Ruatta\thanks{XLIM, France, MATHIS/CRYPTIS and Inria Bordeaux - Sud-Ouest Research Center, France \texttt{olivier.ruatta@unilim.fr}}
}

\begin{document}

\title{Linearized Polynomial Chinese Remainder codes}




\maketitle

\begin{abstract}

We introduce a new family of rank-metric and sum-rank-metric codes called linearized Chinese Remainder Theorem codes ($q$CRT codes), constructed over linearized polynomial rings via a non-commutative generalization of the classical Chinese Remainder Theorem. After establishing the necessary algebraic foundations — including an effective CRT and its lifting for linearized polynomials — we present explicit code constructions and show that several well-known code families, such as Gabidulin or simple codes are related to $q$CRT. A probabilistic decoding algorithm is proposed for a subclass of these codes whose moduli have coefficients in a base field, with an explicit analysis of its failure rate under a uniform error model. This algorithm is further extended to a broader class of codes with moduli over small extension fields. The decoding strategy, inspired by the Chinese Remainder lifting algorithm, exploits the structure of the error support to recover the transmitted codeword. Numerical experiments illustrate the success probability as a function of the error rank weight and code parameters, highlighting the flexibility of the construction and the role of the extension degree in controlling decoding performance.
\end{abstract}

\section{Introduction} \label{intro}

\indent \textbf{Motivation:} Rank and sum-rank-metric codes, such as Gabidulin \cite{gabidulin_theory_1985} codes and LRPC codes \cite{gaborit_low_2013}, have found numerous applications in telecommunications, distributed storage, and cryptography \cite{bartz_rank-metric_2022} and \cite{martinez-penas_codes_2022}. However, the diversity of these codes is limited by their constructions, which often rely on specific base codes or techniques. This paper introduces a new family of rank and sum-rank-metric codes called \textit{linearized Chinese Remainder Theorem (CRT) codes} or \textit{$q$-CRT codes}, along with an efficient decoding algorithm, expanding the scope of potential applications in these domains.

Chinese Remainder Theorem (CRT) codes, introduced by Stone \cite{stone_multiple-burst_1963}, have been extensively studied for both integer and polynomial settings over finite fields. Many works have focused on improving their decoding algorithms, such as those in \cite{goldreich_chinese_1999}, \cite{guruswami_soft-decision_2000}  and \cite{yu_polynomial_2012}\color{black}. The fundamental idea behind CRT codes is simple yet powerful: to encode an element of smaller size (the Euclidean gauge) than the least common multiple of a set of moduli, using the difference in size as redundancy. This approach naturally recovers a wide range of classical code constructions, including Reed-Solomon codes \cite{reed_polynomial_1960} as a special case.

In this work, we apply a similar idea to a special case of right-Euclidean rings, taking advantage of the unique local error properties of rank and sum-rank metrics. This enables us to design a decoding algorithm directly derived from the Chinese Remainder Lifting algorithm,  as in \cite{yu_polynomial_2012}\color{black}.

\textbf{Contributions:} The main contributions of this work are as follows:

\begin{enumerate}
      \item We introduce the mathematical foundation for our new code family: linearized polynomial rings over finite fields and an effective Chinese \\Remainder Theorem tailored to these rings. This enables us to construct a wide range of linearized CRT codes with diverse parameters.
      \item We present explicit constructions and examples of linearized CRT codes, highlighting their relation to known codes in the literature.
      \item We propose a decoding algorithm for a special case of linearized CRT codes and extend it to a broader range of applications. This algorithm has an associated failure rate, which we analyze and discuss.
      \item Finally, we study the parameters of these new codes to provide insights into their decoding properties and potential applications in various fields.
\end{enumerate}

\textbf{Organization:} The remainder of this paper is structured as follows. Section \ref{introrank} introduces rank and sum-rank-metric codes and their applications. In Section \ref{linpol}, we discuss linearized polynomial rings over finite fields and the effective Chinese remainder Theorem for these rings. Section \ref{QPCRT} presents the construction of linearized CRT codes, along with examples and connections to existing codes. A decoding algorithm is proposed in Section \ref{simpdecod}, followed by an analysis of its failure rate and parameters. We extend this algorithm to a wider class in Section \ref{widerdecod} together with the parameters and failure rate in this case.

\section{Rank and sum-rank-metric codes} \label{introrank}

In this section, we introduce the notation and foundational concepts for rank and sum-rank-metric codes. For a more extensive presentation of rank-metric codes, we refer to \cite{bartz_rank-metric_2022}.\\

\noindent Let $q$ be a power of a prime and let $m \in \mathbb{N}$, we denote
$\mathbb{F}_q$ the field with $q$ elements and $\mathbb{F}_{q^m}$ an extension
of $\mathbb{F}_q$ of degree $m$.

\noindent The field $\mathbb{F}_{q^m}$ can be considered as a $\mathbb{F}_q$-vector
space of dimension $m$, i.e. $\mathbb{F}_{q^m} \cong (\mathbb{F}_q)^m$. We
denote $\{ b_1, \cdots, b_m \}$ a basis of $\mathbb{F}_{q^m}$ over
$\mathbb{F}_q$. If $v = z_1 \cdot b_1 + \cdots + z_m \cdot b_m \in
      \mathbb{F}_{q^m}$, we denote $\mathbf{v}= (z_1, \cdots, z_m) \in
      (\mathbb{F}_q)^m$. Let $V = (\mathbb{F}_{q^m})^n$ and let $\zeta = (v_1,
      \cdots, v_n) \in V$ with $v_i = z_{1, i} \cdot b_1 + \cdots + z_{m, i} \cdot
      b_m$ then, we define:
\[ M_{\zeta} = \left(\begin{array}{ccc}
            z_{1, 1} & \cdots & z_{1, n} \\
            \vdots   &        & \vdots   \\
            z_{m, 1} & \cdots & z_{m, n}
      \end{array}\right) . \]
The map $\mathcal{M}: \zeta \in V \longmapsto \mathcal{M}_{\zeta} \in
      \mathcal{M}_{m \times n} (\mathbb{F}_q)$ is a linear isomorphism in a way that
every subspace $W$ of $V$ can be identified to a subspace $\mathcal{M} (W)$ of matrices. Conversely, every $\mathbb{F}_q$-linear subspace in
$\mathcal{M}_{m \times n} (\mathbb{F}_q)$ can be identified with a  $\Fq$-subvector space of $V$.

\begin{definition}
      Rank Metric: Let $V$ and $W$ be two finite dimensional $\mathbb{F}_q$-vector
      spaces and let $c \in End (V, W)$, we define the rank weight $w_r (c)$
      of $c$ as its rank considered as an endomorphism, i.e. $w_r (c) :=
            \tmop{rank} (c)$. For two linear maps $c$ and $c' \in \tmop{End} (V, W)$, the rank distance between them is defined as $d_r (c, c') : = w_r (c - c')$.
\end{definition}

We refer to \cite{bartz_rank-metric_2022} for the following theorem:

\begin{theorem}
      The map $d_r$ is a distance on $\tmop{End} (V, W)$.
\end{theorem}

We now define rank-metric codes:

\begin{definition}
      A rank-metric code is a subset of $\mathbb{F}_q$-linear morphisms, i.e.
      $\mathcal{C}$ is a rank-metric code if there exist $V$ and $W$ two finite
      dimension $\mathbb{F}_q$-vector spaces such that $\mathcal{C} \subset
            \tmop{End} (V, W)$ considered with the rank distance. The code $\mathcal{C}$
      is said to be linear if it is a linear subspace of $\tmop{End} (V, W)$.
\end{definition}

\begin{definition}
      A rank-metric code is said to be matricial if it is a subvector space of a matrix space over $\Fq$.\\
      A rank-metric code is said to be vectorial if it is a subvector space of $\Fqm^n$. In this case, this is also an $\Fqm$-vector space.
\end{definition}

Remark that if $C\subset \mathbb{F}_{q^m}^n$ is a vector rank-metric code of dimension $k$ over $\mathbb{F}_{q^m}$, then $\mathcal{M}(C)$ is a matrix rank-metric code of dimension $mk$ over $\mathbb{F}_q$, and the map $\mathcal{M}$ keeps track of $\Fqm$ linearity. The $\Fqm$ linearity is the core idea of the "Overbeck's attack" on cryptosystems based on vectorial codes, see \cite{overbeck2008structural}. \\
\color{black}

\noindent The notion of support introduced here is very different from the notion of support in Hamming metric. For a vectorial rank-metric code:

\begin{definition}
      Let $\mathbf{v}=(v_1,\cdots,v_n) \in \Fqm^n$ we define the support of $\mathbf{v}$ to be $ supp(\mathbf{v})=\langle v_1,\cdots , v_n\rangle \subset \Fqm$, the subvector space of $\Fqm$ spanned by the coefficients of $\mathbf{v}$.
\end{definition}

In the case of a matricial rank-metric code we define:

\begin{definition}
      Let $\mathbf{m} \in \mathcal{M}_{m \times n}(\Fq)$ the support of $\mathbf{m}$ is the vector space spanned by its columns, i.e. $ supp(\mathbf{m})=Im(\mathbf{m})$.
\end{definition}

This notion of support is central for the decoding algorithm and the weight of a codeword is the dimension of its support as it is summarize in following lemma:
\begin{lemma}
      Let $\mathcal{C} \subset \Fqm^n$ be a rank-metric code and let $\mathbf{c} \in \mathcal{C}$, then $w_r(\mathbf{c})=dim_{\Fq}( supp(\mathbf{c}))$.
\end{lemma}

The proof follows straightforwardly from the definitions of support and rank weight.

\noindent We now extend the notion of rank-metric codes to a more general setting where codewords are partitioned into blocks, each evaluated under the rank metric.

\noindent The sum-rank metric generalizes both the Hamming and rank metrics: it reduces to the Hamming metric when all blocks have size 1 and to the rank metric when there is only one block. Each codeword consists of a set of blocks, each block equipped with the rank
metric and the weight of a code word is the sum of the rank weight of its
blocks. In the case where all the blocks has size one, we have the classical
Hamming weight and if there is only one block, we have to do with the rank
weight. For a general presentation of sum-rank metric, we refer to \cite{martinez-penas_codes_2022}.

\begin{definition}
      Let $V$ and $W_1, \cdots, W_l$ be some finite dimensional
      $\mathbb{F}_q$-vector spaces and let $c = (c_1, \cdots, c_l) \in
            \tmop{End}_{\mathbb{F}_q} \left( V, \underset{i = 1}{\overset{l}{\oplus}}
            W_i \right) = \underset{i = 1}{\overset{l}{\oplus}}
            \tmop{End}_{\mathbb{F}_q} (V, W_i)$ where $c_i \in \tmop{End}_{\mathbb{F}_q}
            (V, W_i)$ then we define the sum-rank weight of $c$, $w_{s r} (c) =
            \overset{l}{\underset{i = 1}{\sum}} w_r (c_i)$. If $c' \in
            \tmop{End}_{\mathbb{F}_q} \left( V, \underset{i = 1}{\overset{l}{\oplus}}
            W_i \right)$ we define $d_{s r} (c, c') = w_{s r} (c - c')$.
\end{definition}

As the sum of distances is a distance, we have:

\begin{theorem}
      The map $d_{s r}$ is a distance on $\tmop{End} \left( V, \underset{i =
                  0}{\overset{l}{\oplus}} W_i \right)$.
\end{theorem}

\noindent While the codes introduced in this work are inherently rank-metric codes, they are even more fundamentally characterized as sum-rank-metric codes.\\

In this work, we consider an adapted error model suitable for rank and sum-rank-metric. When we say that an error $\mathbf{e} \in \Fqm^n$ of weight $r$ is uniform, we mean that $\mathbf{e}$ is chosen uniformly at random from the set of elements in $\Fqm^n$ having (rank or sum-rank) weight $r$. For instance, a uniform error of rank $r$ is a word in $\Fqm^n$ choosen uniformly at random among all  words with support of dimension $r$.
It is a classical assumption in cryptography (see for instance\cite{aragon_minrank_2025-1}, \cite{gaborit_low_2013}).
\color{black}

\section{Linearized polynomials} \label{linpol}

In this section, we recall results on univariate linearized polynomials and linearized polynomial rings, building upon foundational work by \"{O}re on non-commutative polynomials, see \cite{ore_special_1933} and  \cite{ore_theory_1933}. For a comprehensive presentation of these algebraic structures, we refer interested readers to the works of Ruatta \cite{ruatta_polynomes_2022} and Caruso \cite{caruso_polynomes_nodate}.
The ring of linearized polynomials is isomorphic to a skew polynomial ring, see for example \cite{GluesingLuerssen}. Algebraic properties of skew polynomial
rings can therefore be directly translated into the setting of linearized polynomials. The Euclidean division, the algorithm to compute the right greatest
common divisor, as well as related results, follow from the general theory on skew polynomials see \cite{ore_special_1933}, \cite{ore_theory_1933} and \cite {jacobson2009}.\color{black}

To facilitate our subsequent developments, we state and prove two key results: the Chinese Remainder Theorem (CRT) and the lifting of the CRT, tailored to linearized polynomial rings.   Although the Chinese Remainder Theorem is well established, the lifting using Bézout relations and specifically the required hypotheses are not explicit in the literature. \color{black} These constructive statements enable us to efficiently manipulate and construct elements within these rings, which are essential for the construction of our new code family.

We consider
\[
      \mathbb{F}_{q^m}\langle X^q \rangle
      :=
      \left\{
      p(X) \in \mathbb{F}_{q^m}[X]
      ~|~
      \exists d \in \mathbb{N},\exists
      p_0,\dots,p_d \in \mathbb{F}_{q^m}
      \text{ s.t. }
      p(X) = \sum_{i=0}^d p_i X^{q^i}
      \right\}.
\]
which is the
sub-$\tmname{$\mathbb{F}_{q^m}$}$-vector space of $\mathbb{F}_{q^m} [X]$
generated by the monomials $X^{q^i}$ with $i \in \mathbb{N}$. We call this set the set of linearized polynomials over $\Fqm,$ or the set of $q$-polynomials over $ \Fqm$. For all $i \in
      \mathbb{N}$ the map $a \in \mathbb{F}_{q^m} \mapsto a^{q^i}$ is a
$\mathbb{F}_q$-linear map and it will play an important role later. At first, it will be used to motivate the fact to consider composition on
linearized polynomials.

\noindent If we denote by $\theta^j : \left\{\begin{array}{l}
            \mathbb{F}_{q^m} \longrightarrow \mathbb{F}_{q^m} \\
            \zeta \longmapsto \zeta^{q^j}
      \end{array}\right.$ for every $j \in \mathbb{N}$, then for every $\zeta \in
      \mathbb{F}_{q^m}$ and every $P = \underset{i = 0}{\overset{d}{\sum}} p_i \cdot
      X^{q^i} \in \mathbb{F}_{q^m} \langle X^q \rangle$ we define $P (\zeta) =
      \underset{}{\overset{d}{\underset{i = 0}{\sum}} p_i \cdot \theta^i (\zeta) =
            \overset{d}{\underset{i = 0}{\sum}} p_i \cdot \zeta^{q^i}}$.

\begin{lemma}
      \label{linmap}Let $P \in \mathbb{F}_{q^m} \langle X^q \rangle$, then the map
      $\zeta \longmapsto P (\zeta)$ is a $\mathbb{F}_q$-linear map.
\end{lemma}

\noindent It is a simple consequence of the fact that the Frobenius map $\theta=\theta^1$ is linear. If $P \in \mathbb{F}_{q^m}
      \langle X^q \rangle$, we still denote $P$ for the associated linear map to
$P$ by Lemma \ref{linmap}.

We define the support of a linearized polynomial as the support of its coefficients vector.

\subsection{Ring structure}

We consider the composition $\circ : \left\{\begin{array}{l}
            \mathbb{F}_{q^m} \langle X^q \rangle \times \mathbb{F}_{q^m} \langle X^q
            \rangle \longrightarrow \mathbb{F}_{q^m} \langle X^q \rangle \\
            (P, Q) \longmapsto P \circ Q
      \end{array}\right.$ which can be considered as a product induced by $(a \cdot
      X^{q^i}) \circ (b \cdot X^{q^j}) = a \cdot b^{q^i} \cdot X^{q^{i + j}}$. The
set $(\mathbb{F}_{q^m} \langle X^q \rangle, +, \circ)$ is a non-commutative
$\mathbb{F}_{q}$-algebra. We will denote by $(P)_l$ the left ideal generated by an element $P$. \color{black} We fix some notations, allowing us to see that this
algebra is a right Euclidean ring (both on the right and left side, but we will
focus on division on the right).

\noindent Let $P = \underset{i = 0}{\overset{d}{\sum}} a_i \cdot X^{q^i} \in
      \mathbb{F}_{q^m} \langle X^q \rangle$, this linearized polynomial is said to
be of $q$-degree $d$ if $a_d \neq 0$ and we denote $\deg_q(P) = d$. We
assume now that $\deg_q(P) = d$, then its leading term $\tmop{LT} (P)$
is $a_d \cdot X^{q^d}$, its leading coefficient $\tmop{LC} (P)$ is $a_d$ and
its leading monomial is $\tmop{LM} (P) = X^{q^d}$. To define Euclidean
division on the right between linearized polynomials, we first explain how it
works for terms just like the product. Let $i > j$ be two integers and $a$ and
$b \in \mathbb{F}_{q^m}$ with $b \neq 0$, we have $a \cdot X^{q^i} = c \cdot
      X^{q^{i - j}} \circ b \cdot X^{q^j} = c \cdot b^{q^{i - j}} \cdot X^{q^i}$ and
then we need to have $a = c \cdot b^{q^{i - j}}$ in a way that if $c = a \cdot
      b^{q^{j - i}}$ then $a \cdot X^{q^i} = c \cdot X^{q^{i - j}} \circ b \cdot
      X^{q^j}$ and we denote $c \cdot X^{q^{i - j}} = \tmop{rquo} (a \cdot X^{q^i},
      b \cdot X^{q^j})$ the quotient of the right division of $a \cdot X^{q^i}$ by
$b \cdot X^{q^j}$. This allows us to define an Euclidean right division
between linearized polynomials. We give now an algorithm for the right division
of a linearized $A$ and a linearized polynomial $B$.

\begin{algorithm}
      \caption{Algorithm for $q$-polynomial right  division (RQUOREM)}
      \label{rqorem}
      \begin{algorithmic}[1]
            \STATE \bf{Input:} $A$ and $B$ of $q$-degree respectively $d_A$ and $d_B$.
            \STATE \bf{Output:} $Q$ and $R$ such that $A = Q \circ B + R$ where the $q$-degree of $R$ is strictly less than the $q$-degree of $B$.
            \STATE $R \gets A$
            \STATE $Q \gets 0$
            \STATE $d_R \gets d_A$
            \WHILE{$d_B \leqslant d_R$}
            \STATE $Q \gets Q + \operatorname{LC}(R) \cdot \operatorname{LC}(B)^{q^{d_B - d_R}} \cdot X^{q^{d_R - d_B}}$
            \STATE $R \gets R - \bigl( \operatorname{LC}(R) \cdot \operatorname{LC}(B)^{q^{d_B - d_R}} \cdot X^{q^{d_R - d_B}} \bigr) \circ B$
            \STATE $d_R \gets \deg_q(R)$
            \ENDWHILE
            \RETURN $(Q,R)$
      \end{algorithmic}
\end{algorithm}

\begin{lemma}
      The result of the algorithm $\text{rquorem} (A, B)$ is a pair $(Q, R)$ such
      that $A = Q \circ B + R$ with $\deg_q(R)< \deg_q(B)$ and if
      $R = 0$ then $B$ divides $A$ on the right.
\end{lemma}

\noindent The right division plays a special role with respect to the associated
$\mathbb{F}_q$-linear map associated to a linearized polynomial by Lemma
\ref{linmap}.

\begin{lemma}
      \label{lemeval}Let $\zeta \in \mathbb{F}_{q^m}$ then for all $P \in
            \mathbb{F}_{q^m} \langle X^q \rangle$ the remainder of the right
      division of $P$ by $X^q - \zeta^{q - 1} \cdot X$ is equal to $ \zeta^{-1}\cdot P (\zeta)
            \cdot X$.
\end{lemma}

\begin{corollary}
      If $P (\zeta) = 0$ or equivalently if $\zeta \in \ker (P)$ then $X^q -
            \zeta^{q - 1} \cdot X$ divides $P$ on the right.
\end{corollary}

\noindent As a consequence of the existence of the right Euclidean division we have :

\begin{proposition}
      The ring $(\mathbb{F}_{q^m} \langle X^q \rangle, +, \circ)$ is right
      Euclidean and then:
      \begin{enumerate}
            \item Every left ideal is principal ($\mathbb{F}_{q^m} \langle X^q
                        \rangle$ is right principal).

            \item If $A$ and $B \in (\mathbb{F}_{q^m} \langle X^q \rangle, +, \circ)$
                  then $A$ and $B$ admit a greatest right common divisor denoted $A \wedge_r
                        B$ and a left least common multiple $A \vee_l B$.
      \end{enumerate}
\end{proposition}

\noindent We will see that both right greatest common divisor (denoted RGCD) and left
least common multiple (denoted LLCM) can be computed by a generalization of the
extended Euclide algorithm in the non-commutative case. For the sake of
simplicity, we will assume that $A$ and $B \in \mathbb{F}_{q^m} \langle X^q
      \rangle$ and that $\deg_q(A) \geqslant \deg_q(B)$. The
following algorithm compute the coefficients of the B{\'e}zout relation for
the $\tmop{RGCD}$.

\begin{algorithm}
      \caption{Extended right Euclidean algorithm for $q$-degree polynomials (RGCD)}
      \begin{algorithmic}[1]
            \REQUIRE $A, B \in \mathbb{F}_{q^m} \langle X^q \rangle$
            \ENSURE $U_0, V_0, R_0$ such that $U_0 \circ A + V_0 \circ B = R_0 = A \wedge B$
            \STATE $U_0 \gets 1,\; V_0 \gets 0,\; U_1 \gets 0,\; V_1 \gets 1,\; R_0 \gets A,\; R \gets B,\; Q \gets 0,\; U_t \gets 0,\; V_t \gets 0$
            \WHILE{$R \neq 0$}
            \STATE $(Q, R) \gets \text{rquorem}(R_0, R)$
            \STATE $U_t \gets U_1,\; V_t \gets V_1$
            \STATE $U_1 \gets U_0 - Q \circ U_1$
            \STATE $V_1 \gets V_0 - Q \circ V_1$
            \STATE $U_0 \gets U_t,\; V_0 \gets V_t$
            \STATE $R_0 \gets R$
            \ENDWHILE
            \RETURN $(U_0, V_0, R_0)$
      \end{algorithmic}
\end{algorithm}

\noindent In this algorithm, the third entry $R_0$ is $A \wedge_r B$. A slight change
on the output of the algorithm allows to get the coefficient of the B{\'e}zout
relation for the LLCM:

\begin{algorithm}[H]
      \caption{Algorithm for computing left co-multipliers (LLCM)}
      \begin{algorithmic}[1]
            \REQUIRE $A, B \in \mathbb{F}_{q^m} \langle X^q \rangle$
            \ENSURE $U_1, V_1$ such that $U_1 \circ A + V_1 \circ B = 0$ and
            $\deg_q(U_1) = \deg_q(B) - \deg_q(A \wedge_r B)$ and
            $\deg_q(V_1) = \deg_q(A) - \deg_q(A \wedge_r B)$.
            \STATE $U_0 \gets 1,\; V_0 \gets 0,\; U_1 \gets 0,\; V_1 \gets 1,\; R_0 \gets A,\; R \gets B,\; Q \gets 0,\; U_t \gets 0,\; V_t \gets 0$
            \WHILE{$R \neq 0$}
            \STATE $(Q, R) \gets \text{rquorem}(R_0, R)$
            \STATE $U_t \gets U_1,\; V_t \gets V_1$
            \STATE $U_1 \gets U_0 - Q \circ U_1$
            \STATE $V_1 \gets V_0 - Q \circ V_1$
            \STATE $U_0 \gets U_t,\; V_0 \gets V_t$
            \STATE $R_0 \gets R$
            \ENDWHILE
            \RETURN $(U_1, V_1)$
      \end{algorithmic}
\end{algorithm}

Then $A \vee_l B = U_1 \circ A = -V_1 \circ B$.\\

\subsubsection{Complexity of basic arithmetic over linearized polynomials} \label{complex1}

Here, we  summarize few results on the complexity of the different arithmetic operations on linearized
polynomials presented in this section. These results will be needed in order to
evaluate the complexity of the algorithms for Chinese remainder lifting. We
refer to the work of Caruso and Le Borgne \cite{CARUSOLEBORGNE}. Here $\SM{d}{m}$ denotes the number of operations in
$\Fq$ required to multiply two polynomials in $\QP$ of $q$-degree $d$. We assume that all operations
in $\Fqm$ can be performed in quasi-linear time, that is, all arithmetic operations in $\Fqm$ have complexity
$\tilde{\mathcal{O}}(m)$ operations in $\Fq$, where $\tilde{\mathcal{O}}$ ignores
polylogarithmic factors. We focus on the complexity of the basic arithmetic in $\QP$.
Clearly, addition has linear complexity. Multiplication plays a central role in the arithmetic.
The multiplication of a linearized polynomial of $q$-degree $d_1$ by a linearized polynomial of $q$-degree $d_2$
can be performed ``naively" using $\sO(d_2\cdot m^2 +d_1\cdot d_2 \cdot m)$ operations in $\Fq$. In particular, if $d_1$ and
$d_2$ are both bounded by $d$, the product can be computed using
$\sO(d^2 \cdot m + m^2 \cdot d)$ operations in $\Fq$.  From \cite{CARUSOLEBORGNE}, using
Karatsuba-like approach, the multiplication of two linearized polynomials of $q$-degree at most $d$
can be done using $\sO(d^{1.58} \cdot m^{1.41})$ operations in $\Fq$. In other words,
$\SM{d}{m} \in \sO(d^{1.58} \cdot m^{1.41})$ . The division of a linearized
polynomial of $q$-degree at most $d$ by a linearized polynomial of $q$-degree at most $d$
can be performed in $\sO(\SM{d}{m})$ operations in $\Fq$. Using algorithm 6 of \cite{CARUSOLEBORGNE},
the right greatest common divisor of two linearized polynomials of $q$-degree bounded by $d$ can be done using
$\sO(\SM{d}{m})$  operations in $\Fq$ (an amount of $\log(d)$ euclidian divisions using a divide and conquer approach).
This last complexity is important in order to evaluate the complexity of the lifting of the Chinese remainder theorem.

\begin{remark} \label{remult}
      $\SM{l \cdot d}{m} \leq l^2 \cdot \SM{d}{m}$.
\end{remark}

\color{black}

\subsection{Link with endomorphisms of $\mathbb{F}_{q^m}$}\label{endo}

We have already seen that one can interpret every linearized polynomial as an
$\mathbb{F}_q$-endomorphism of $\mathbb{F}_{q^m}$ and so Lemma \ref{linmap}
gives an embeding of $\mathbb{F}_{q^m} \langle X^q \rangle$ into the
$\tmop{End}_{\mathbb{F}_q} (\mathbb{F}_{q^m})$ of the
$\mathbb{F}_q$-endomorphism of $\mathbb{F}_{q^m}$. Recalling that
$\mathbb{F}_{q^m}$ is the splitting field of $X^{q^m} - X$ one can deduce
that $P$ defines the same linear map as its remainder by the right division by
$X^{q^m} - X$. It is summarized in the following lemma.

\begin{lemma} \label{lemendo}
      The sets $\frac{\QP}{\left( X^{q^m}-X\right)_l}$ and $\tmop{End}_{\mathbb{F}_q} (\mathbb{F}_{q^m})$ are isomorphic as $\Fq$-algebras.
\end{lemma}

\subsection{Effective Chinese Remainder Theorem for linearized
      polynomials}\label{secECRT}

In this subsection, we focus on an algorithm to reconstruct a linearized
polynomial from its remainders with respect to several moduli linearized
polynomials. In \cite{R1reviewer1}, authors give a generalization of
Lagrange interpolation to linearized polynomials. Lagrange interpolation for classical
polynomials is a particular case of the Chinese Remainder Theorem for classical polynomials (see Theorem 1 in \cite{yu_polynomial_2012}), when all the moduli have degree $1$.
The algorithm given in this subsection provides a full generalization of Theorem 1 from  \cite{yu_polynomial_2012} to the case of linearized polynomials.  \color{black}

We first give the result for two linearized moduli polynomials
and then we show how to generalize that for an arbitrary number of moduli
polynomials.

Let $f_1$ and $f_2 \in \mathbb{F}_{q^m} \langle X^q \rangle$ of $q$-degree
repsectively $d_1$ and $d_2$. For any linearized polynomial $g$ we
denote $\pi_i (g)$ its remainder with respect to the right division by $f_i$. We
assume that $f_1$ and $f_2$ are coprime, that is to say that we have a
B{\'e}zout relation $S_1 \circ f_1 + S_2 \circ f_2 = X$ with $\deg_q
      (S_1) < d_2$ and $\deg_q (S_2) < d_1$ using the algorithm
$\tmop{RGCD}$. Remark that $S_1$ is the inverse of $f_1$ modulo $f_2$, i.e.
$\pi_2 \left( S_1 {\circ f_1}  \right) = X$ and that $S_2$ is the inverse of
$f_2$ modulo $f_1$, i.e. $\pi_1 (S_2 \circ f_2) = X$.

\begin{lemma}
      \label{lemtech}Let $g \in \mathbb{F}_{q^m} \langle X^q \rangle$ then $g -
            \pi_2 (g) \circ S_1 \circ f_1 - \pi_1 (g) \circ S_2 \circ f_2$ lies in the
      left ideal $(f_1 \vee_l f_2)_l$.
\end{lemma}

\begin{proof}
      Since $S_2$ is the inverse of $f_2$
      modulo $f_1$, we have \[ g - \pi_2 (g) \circ S_1 \circ f_1 - \pi_1
            (g) \circ S_2 \circ f_2  = \pi_1 (g) -
            \pi_1 (g) \circ S_2 \circ f_2 = 0 \mod f_1. \] Therefore, $f_1$ divides
      $g - \pi_2 (g) \circ S_1 \circ f_1 - \pi_1 (g) \circ S_2 \circ f_2$ on the
      right.  In the same way, we show that $f_2$ divides
      $g - \pi_2 (g) \circ S_1 \circ f_1 - \pi_1 (g) \circ S_2 \circ f_2$ on the
      right.  We
      conclude that since $g - \pi_2 (g) \circ S_1 \circ f_1 - \pi_1 (g) \circ S_2
            \circ f_2$ can be divided on the right by both $f_1$ and $f_2$, it a left
      multiple of $f_1 \vee_l f_2$.
\end{proof}

\noindent For all $g \in \mathbb{F}_{q^m} \langle X^q \rangle$, we denote by $\pi_{1,2} (g)$
the remainder by the right division of $g$ by $f_1 \vee_l f_2$. We are now able to
state the Chinese Remainder Theorem for linearized polynomials with two
moduli.

\begin{theorem}
      \label{CRT-2}Let $g \in \mathbb{F}_{q^m} \langle X^q \rangle$ of $q$-degree
      strictly less than $d_1 + d_2$ then $\pi_{1,2} (\pi_2 (g) \circ S_1 \circ f_1 +
            \pi_1 (g) \circ S_2 \circ f_2) = g$.
\end{theorem}

\begin{proof}
      By Lemma \ref{lemtech}, we know that $g - (\pi_2 (g) \circ S_1 \circ f_1 +
            \pi_1 (g) \circ S_2 \circ f_2) \in (f_1 \vee_l f_2)_l$ and since $\deg_q (g) < d_1 + d_2$ we have $\pi_{1,2} (g - (\pi_2 (g) \circ S_1 \circ
            f_1 + \pi_1 (g) \circ S_2 \circ f_2)) = \pi_{1,2} (g) - \pi_{1,2} (\pi_2 (g) \circ
            S_1 \circ f_1 + \pi_1 (g) \circ S_2 \circ f_2) = g - \pi_{1,2} (\pi_2 (g) \circ
            S_1 \circ f_1 + \pi_1 (g) \circ S_2 \circ f_2) = 0.$ Then, we deduce that $g
            = \pi_{1,2} (\pi_2 (g) \circ S_1 \circ f_1 + \pi_1 (g) \circ S_2 \circ f_2)$.
\end{proof}

Using the same notation as Theorem \ref{CRT-2} and assuming that $d_1=d_2=d$, we then have:
\begin{lemma} \label{lemcomp1}
      Given $f_1$, $f_2$, $g_1$ and $g_2$ with $deg_q(g_1)<deg_q(f_1)$ and $deg_q(g_2) < deg_q(f_2)$, the computation of $g$ with $deg_q
            (g)< deg_q(f_1  \vee_l f_2)$ such that $\pi_1(g)=g_1$ and $\pi_2(g)=g_2$ can be done using $\sO(\SM{d}{m})$ operations in $\Fq$.
\end{lemma}

\begin{proof}
      From Theorem \ref{CRT-2}, the polynomial $g= \pi_{1,2} (g_2 \circ S_1 \circ f_1 +  g_1 \circ S_2 \circ f_2)$ satisfies
      $\pi_1(g)=g_1$ and $\pi_2(g)=g_2$. The computation of $S_1$ and $S_2$ is performed using the extended Euclide algorithm on $f_1$ and
      $f_2$ in a way that they can be computed using $\sO(\SM{d}{m})$ operations in $\Fq$.
      The computation of $g_1 \circ S_2 \circ f_2$ and $g_2 \circ S_2 \circ f_2$ can be performed using $\sO(4 \cdot \SM{d}{m})$ by Remark \ref{remult}.
      The same argument holds for the last division of $g_1 \circ S_2 \circ f_2$ and $g_2 \circ S_1 \circ f_1$ (both of $q$-degree $3 \cdot d$) by $f_1 \vee_l f_2$ (which has $q$-degree $2 \cdot d$ )
      that can be computed using $\sO(9 \cdot \SM{d}{m})$, since all involved $q$-polynomials have $q$-degree bounded by $3 \cdot d$ . This allows to conclude.
\end{proof}

\color{black}

\noindent It is important to understand that even if $f_1, f_2$ and $f_3$ are pairwise coprime, then $f_1 \vee_l f_2$ is not always coprime with $f_3$. Assume
$\zeta$ and $\xi \in \mathbb{F}_{q^m}$ and are $\mathbb{F}_q$ independent.
Consider $f_1 = X^q - \zeta^{q - 1} \cdot X$, $f_2 = X^q - \xi^{q - 1} \cdot
      X$ and $f_3 = X^q - (\zeta + \xi)^{q - 1} \cdot X$. Those three linearized
polynomials are coprime. Since $\zeta \in \ker (f_1)$ and $\xi \in \ker
      (f_2)$, $\zeta + \xi \in \ker (f_1 \vee_l f_2)$ and therefore $f_3 = X^q -
            (\zeta + \xi)^{q - 1}  \cdot X$ divides $f_1 \vee_l f_2$ on the right. It explains why
the construction of the non-commutative lifting for the Chinese Remainder
Theorem is not straightfoward in the case of three polynomials or more. \\

\noindent Let $f_1, f_2, \cdots, f_k \in \mathbb{F}_{q^m} \langle X^q \rangle$ of
$q$-degree respectively $d_1, d_2, \cdots, d_k$ and such that $f_1 \vee_l f_2
      \vee_l \cdots \vee_l f_{i - 1}$ is coprime with $f_i$ for all $i \in \{ 2,
      \cdots, k \}$. For all $g \in \mathbb{F}_{q^m} \langle X^q \rangle$, we denote
$\pi_i (g)$ \ the remainder of $g$ by the right division by $f_i$. We define
$h_1 = f_1$, \ $h_{i + 1} = h_i \vee_l f_{i + 1}$ for $i \in \{ 1, \cdots, k -
      1 \}$ and $\pi_{1, i} (g)$ the remainder of $g$ by the right division by $h_i$
for all $i \in \{ 2, \cdots, k \}$. We also consider $S_{i, 1}$ and $S_{i+1, 2}$
the coefficients of the B{\'e}zout relation $S_{i, 1} \circ h_i + S_{i+1, 2}
      \circ f_{i+1} = X$ between $h_i$ and $f_{i+1}$ (since they are assumed to be coprime).

\begin{theorem}
      \label{CRT}With the notations hereabove, if $g \in \mathbb{F}_{q^m} \langle
            X^q \rangle$ then defining $g_1 = \pi_1 (g)$ and $g_i = \pi_i(g) \circ S_{1, i-1} \circ h_{i-1} + g_{i-1} \circ S_{2, i} \circ f_i$, we have
      $g - g_k \in (f_1 \vee_l \cdots \vee_l f_k)_l$ and then if $\deg_q
            (g) < \deg_q (f_1 \vee_l \cdots \vee_l f_k) = d$ then $g_k = g$.
\end{theorem}

\begin{proof}
      We will prove the result by induction. First, we have $g - g_1 \in (h_1)_l$
      since $h_1 = f_1$ and $g = \pi_1 (g)$. If the $q$-degree of $g$ is
      smaller than the one of $f_1$ then $g_1 = g$.   Assume now that there exists $i$ such that for all $j\leqslant i,$ $g - g_j \in
            (h_j)$.

      \noindent We then have \begin{align*} g-g_{i+1} & =\pi_{i+1}(g) - (\pi_{i+1}(g) \circ S_{1, i} \circ h_{i} + g_{i} \circ S_{2, i+1} \circ f_{i+1}) \\
                               & = \pi_{i+1}(g) - \pi_{i+1}(g) = 0   \mod f_{i+1}\end{align*} since $ S_{1, i}$ is the inverse of $h_{i}$ mod $f_{i+1}.$

      \noindent We also have \begin{align*}
            g-g_{i+1} & = g - (\pi_{i+1}(g) \circ S_{1, i} \circ h_{i} + g_{i} \circ S_{2, i+1} \circ f_{i+1}) \\
                      & =g- g_{i} =0
            \mod h_{i}
      \end{align*}  since $ S_{2, i+1}$ is the inverse of  $f_{i+1}$ mod $h_{i}$, and by induction hypothesis.
      This shows that $g-g_{i+1} $ can by divided on the right by both $h_i$ and $f_{i+1}$. Therefore, $g - g_{i+1} \in
            (h_{i+1})$.
\end{proof}

\noindent This theorem leads directly to an algorithm (allowing early termination since
the $f_i$ are introduced incrementally and it also allows to control the
$q$-degree of the linearized polynomials involving during the algorithm) but
it is equivalent to the following statement. We denote $b_i = \underset{j \neq
            i}{\vee_l} f_j$ and we assume that $f_i \wedge_r b_i = X$ (it is equivalent of
the hypothesis of the previous theorem) and we denote $S_{1, i} \circ b_i +
      S_{2, i} \circ f_i = X$ (i.e. $S_{1, i}$ is the inverse of $b_i$ mod $f_i$).

\begin{theorem} \label{CRT2}
      Denoting $G =  \underset{i = 1}{\overset{k}{\sum}} \pi_i
            (g) \circ S_{1,i} \circ b_i $, we have $g - G \in (f_1 \vee_l \cdots
            \vee_l f_k)_l$ and if $\deg_q (g) < d$ then $\pi_{1,k}(G) = g$.
\end{theorem}

\begin{proof}
      Let $i \in \{1, \dots, k \}.$ We have : \[ \pi_i(g -G)  = \pi_i(g)- \pi_i(G)=\pi_i(g) - \pi_i(g)  = 0,  \] since $f_i$ divides $b_j$ for all $i \neq j,$ and $S_{1,i} $ is the inverse of $b_i \mod f_i.$
      Hence for all $i \in \{1, \dots, k \},$ $f_i$ divides $g-G$, and then $f_1 \vee_l \cdots \vee_l f_k$ divides $g-G$ on the right. This means that $g-G \in(f_1 \vee_l \cdots \vee_l f_k)_l$.
      If $\deg_q(g) <\deg_q(f_1 \vee_l \cdots \vee_l f_k$), then $\pi_{1,k} (g)=g,$ and since $\pi_{1,k} (g-G)=0,$ we deduce that $g = \pi_{1,k}(G).$
\end{proof}

In order to analyze the worst-case complexity of the lifting of the Chinese Remainder Theorem,
we assume that all the $f_i$ have $q$-degree $d$ and that for all $i \in \left\{1,\dots,k-1\right\}$,  and letting
$h_i=f_1 \vee_l \dots \vee_l f_i$ such that $h_i \wedge_r f_{i+1}= X$, we have $\deg_q(f_1 \vee_l \dots \vee_l f_k)=k \cdot d$
and more generally $\deg_q(h_i)=i \cdot d$.  Then, at each step $i$, the lifting consists  of applying the two-moduli lifting and using Lemma \ref{lemcomp1}.

\begin{lemma}  \label{lemcomp2} For $i \in \{1,\dots, k-1\}$, at step $i$, the number of operations in $\Fq$ required to perform the lifting is
      dominated by $i^2 \cdot \SM{d}{m}$. \end{lemma}

\begin{corollary} \label{comptot}
      Under the above assumptions, the complexity of the lifting of the Chinese Remainder Theorem is $\sO(k^3 \cdot \SM{d}{m})$.
\end{corollary}
This estimate is likely very pessimistic, but it shows that this algorithm runs in polynomial time in the size of the input.
\color{black}

\section{Linearized Polynomial Chinese Remainder Theorem codes} \label{QPCRT}

In this section, we present the construction of a new family of codes. We
give a general construction and we show how others already known rank-metric
codes can be obtained by this construction. Then, we give some special cases of constructions of new rank-metric codes with interesting properties.

Let $f_1, \cdots, f_s \in \mathbb{F}_{q^m} \langle X^q \rangle$ of $q$-degree respectively $d_1, \cdots, d_s$. We denote $\pi_i : \left\{\begin{array}{l}
            \mathbb{F}_{q^m} \langle X^q \rangle \longrightarrow \mathbb{F}_{q^m}
            \langle X^q \rangle \\
            g \longmapsto \pi_i (g)
      \end{array}\right.$ where $\pi_i (g)$ is the remainder of the right division
of $g$ by $f_i$, for all $i \in \{ 1, \cdots, s \}$. We define the following
map:
\[ \Pi = \pi_1 \times \cdots \times \pi_l : \left\{\begin{array}{l}
            \mathbb{F}_{q^m} \langle X^q \rangle \longrightarrow \mathbb{F}_{q^m}
            \langle X^q \rangle / (f_1)_l \times \cdots \times \mathbb{F}_{q^m}
            \langle X^q \rangle / (f_s)_l \\
            g \longmapsto (\pi_1 (g), \cdots, \pi_s (g))
      \end{array}\right. . \]

The map $\Pi$ is an $\mathbb{F}_q$-linear map. Let $A \in \mathbb{F}_{q^m}
      \langle X^q \rangle$, we define the following map:

\[ \mathcal{M}_A : \left\{\begin{array}{l}
            \mathbb{F}_{q^m} \langle X^q \rangle \longrightarrow \mathbb{F}_{q^m}
            \langle X^q \rangle \\
            g \longmapsto g \circ A
      \end{array}\right. . \]

The map $\mathcal{M}_A$ is also a $\mathbb{F}_q$-linear map, and therefore the
following map is also $\mathbb{F}_q$-linear:
\begin{equation}
      \label{map} \Psi_A = \Pi \circ \mathcal{M}_A : \left\{\begin{array}{l}
            \mathbb{F}_{q^m} \langle X^q \rangle \longrightarrow \mathbb{F}_{q^m}
            \langle X^q \rangle / (f_1)_l \times \cdots \times \mathbb{F}_{q^m}
            \langle X^q \rangle / (f_s)_l \\
            g \longmapsto (\pi_1 (g \circ A), \cdots, \pi_s (g \circ A))
      \end{array}\right. .
\end{equation}

\noindent Let $k \in \mathbb{N}$, we denote $\mathbb{F}_{q^m} \langle X^q \rangle_k$ the set of linearized polynomial of $q$-degree strictly less than $k$.
We now have all the ingredients to give the definition of the $q$CRT codes.

\begin{definition}
      The $q$CRT associated to $k$, $A$ and $F = (f_1, \cdots, f_s)$ is
      $$ \mathcal{C}_{F, k, A} = \Psi_A (\mathbb{F}_{q^m} \langle X^q \rangle_k).$$
\end{definition}

\subsection{Main properties of $q$CRT codes}

In this section, we will always assume that we satisfy hypotheses of Theorem \ref{CRT2}. We consider $F = (f_1, \cdots, f_s) \in
      \mathbb{F}_{q^m} \langle X^q \rangle^s$ where $f_i$ is coprime with $h_i = f_1
      \vee_l \cdots \vee_l f_{i - 1}$ for all $i \in \{ 2, \cdots, s \}$ (so
that the full Chinese Remainder Theorem can be applied). We assume that $f_i$ is
of $q$-degree $d_i$ for all $i \in \{ 1, \cdots, s \}$ and we denote $n = d_1
      + \cdots + d_s$. Assume that $A$ has $q$-degree $\alpha$, i.e. $A =
      \underset{i = 0}{\overset{\alpha}{\sum}} a_i \cdot X^{q^i}$.

\begin{proposition}
      Let $F$ and $n$ defined as above.
      Consider $k < n - \alpha$. The $q$CRT code associated to $F$ and $k$ $\mathcal{C}_{F, k,A}$
      is a $\mathbb{F}_{q^m}$-linear rank-metric code of dimension $k$ and length $n$.
\end{proposition}
\color{black}

\begin{remark}
      A codeword in the $q$CRT code $\mathcal{C}_{F, k,A}$ is a set of $s$ remainders, represented by $s$ linearized polynomials. Each remainder of index $i$ has $q$-degree bounded by $d_i$. Therefore, the length $n$ defined here corresponds to the weighted $q$-degree of $F$. This is identical to the length $N$ defined in \cite{yu_polynomial_2012}, which was used to decode Polynomial Remainder codes of \cite{yu_polynomial_2012} in the weighted distance. \\
      Remark that as in \cite{yu_polynomial_2012}, we allow code symbols to have different degrees. \\
      The dimension defined here is related to the dimension $K$ defined in \cite{yu_polynomial_2012}, but is more general, since the dimension of $\mathcal{C}_{F, k,A}$ is not necessarily equal to the weighted $q$-degree of a subset of polynomials contained in $F$.
\end{remark}
\color{black}

We will consider $q$CRT codes in the rank metric, by seeing a codeword as an element of $\Fqm^n$.
\color{black}

The polynomial $A$ has an important role in controlling the minimum distance of the code. For instance if $A=X$ then the multiplication by $A$ is the identity and we always have a rank-weight $1$ word in the code (by encoding $P(X)=X$).

We conjecture that with reasonable hypothesis on the $q$-degree of the moduli linearized polynomials, the minimum distance is the weight of $A(X)$, i.e. the dimension of its support. Numerous tests have been done to reinforce this heuristic and it can be shown in some simple cases.

Conversely, if $A$ has a support of dimension $r$, for every word $\mathbf{z}$ in the code $\mathcal{C}_{F,k,A}$ there exists $P \in \QP_k$ such that $\mathbf{z} = \Psi_A(P)$ in a way that $\mathbf{z}$ is composed of $s$ blocks of coefficients of the linearized polynomials $ P \circ A$ divided on the right by $f_i$ for $i \in \{1,\cdots,s\}$. If the sum of the $q$-degree of the $f_i$ is greater than $r$, one can interpret the division on the right by the $f_i$ as a $\Fq$ linear transform on the coefficients of $P \circ A$. This leads us to the following assertion:  if the $f_i$ have $q$-degree large enough, since by definition  $d_{min} = \min\{ \dim_{\Fq}( \sum_{i=1}^s \text{supp}(P \circ A \, \text{ mod}_r \, f_i)) | P \in \QP_k \}$, we conjecture that:
\begin{align*}
      d_{min} & =\min\{\dim_{\Fq}(\sum_{i=1}^s \text{supp}(P \circ A \, \text{ mod}_r \, f_i)) | P \in \QP_k\} \\& =  \min\{ \dim_{\Fq}(\text{supp}(P \circ A)) | P \in \QP_k \} \\&=\dim_{\Fq}(\text{supp}(A)).
\end{align*}
\color{black}

The required properties on $A$ are not restrictive and it is easy to find a suitable choice for $A$. It is possible to build such a polynomial at random. Building a polynomial $A$ of degree $\alpha$ and rank weight $r$  with coefficients in $\mathbb{F}_{q^m}$ can be done by choosing at random a vector in $\mathbb{F}_{q^m}^{\alpha+1}$ of rank weight $r$, with a non zero coefficient in the last position (in order to indeed build a polynomial of degree $\alpha$), and then building a polynomial whose coefficients are the coefficients of this vector.

\begin{remark}
      This remark is devoted to the link with sum-rank metric and related works using sum-rank metric or sum-rank weight. Let $\mathcal{C}_{F,k,A}$ be a linearized Chinese remainder code associated to $F=(f_1,\cdots,f_s)$. There is a natural way to consider this code in the sum-rank metric. Using definitions and notations above, if ${\bf z} \in  \mathcal{C}_{F,k,A}$ then there exists $P\in \QP_k$ such that ${\bf {z}}=\Psi_A(P) = (\pi_1(P \circ A),\cdots,\pi_s(P \circ A))$ and we define the sum-rank weight of ${\bf z}$ as $$\omega_{sr}({\bf{z}})=\displaystyle \sum_{i=1}^{s} \omega_r(\pi_i(P \circ A)),$$ where $\omega_r(\pi_i(P \circ A))$ is the rank weight of $\pi_i(P \circ A)$, that is, the rank of the matrix built from the vector of the coefficients of $\pi_i(P \circ A)$ via the map $\mathcal{M}$ of Section \ref{introrank}.
      This point of view is close to the one in \cite{martinez-penas_2018} over some special skew polynomials defining skew and linearized Reed-Solomon codes, but the blocks of those codes are linked to some conjugacy classes. If the construction given here can be viewed as a generalization of the work of Mart\'{i}nez-Pe\~{n}as, this is not very straightfoward and a specific work is needed to give a complete undertanding on this perspective.
\end{remark}
\color{black}

\subsection{Some special constructions and examples}
\subsubsection{Link with Gabidulin codes}

A well known family of codes linked to the Gabidulin codes can be obtained as special
case of $q$CRT codes: \\
Let $\zeta_1, \cdots, \zeta_n \in \mathbb{F}_{q^m}$ be
$\mathbb{F}_q$-independent. The polynomials $f_i := X^q -
      \zeta_i^{q - 1} \cdot X$ for $i \in \{ 1, \cdots, n \}$ satisfy the conditions
of the previous subsection. According to Lemma \ref{lemeval}, the
remainder of the division of $g \in \mathbb{F}_{q^m} \langle X^q
      \rangle_k$ by $f_i$ is $g (\zeta_i) \cdot \zeta_i^{- 1} \cdot X$. Using this
fact, we see that the generator matrix of $\mathcal{C}_{F, k,X}$ is the one of the
Gabidulin code associated to $\mathbf{\zeta} = (\zeta_1, \cdots, \zeta_n)$ of dimension $k$ and
length $n$, multiplied by the diagonal matrix $\tmop{diag} (\zeta_1^{_{- 1}},
      \cdots, \zeta_n^{- 1})$. \\

\noindent These codes were already studied. Indeed, these codes can be seen as skew Reed-Solomon codes.

\subsubsection{An instance with efficient encoding and efficient lifting}

Let $l <m$ and $F=(f_1, f_2)\in (\QP)^2$ with $f_1 = X^{q^l}$ and $f_2=X^{q^m}-X$. Since $ X^{q^l} \circ X^{q^{m-l}}-(X^{q^m}-X)\circ X = X$, this gives a Bézout relation proving that $f_1$ and $f_2$ are coprime. Let $A\in \QP,$ and $k<m+l+\deg_q(A).$ Consider $\mathcal{C}_{F,k,A}$. \\
Denote $\mathcal{P}_l : \left\{\begin{array}{l}
            \QP \rightarrow \QP_{l} \\
            \underset{i = 0}{\overset{m}{\sum}} a_i \cdot X^{q^i} \mapsto \underset{i =
                                                                                0}{\overset{l-1}{\sum}} a_i \cdot X^{q^i}
      \end{array}\right.$ the projection on $\QP_{l}$. \\
First remark that the remainder of any $P \in \QP$ with respect to the right division by $X^{q^l}$ is $\pi_1(P)=\mathcal{P}_l(P)$.
Then, each codeword of $\mathcal{C}_{F,k,A}$ is of the form $(\mathcal{P}_l(P\circ A), \pi_2(P\circ A)),$ with $P \in \QP_{k}$, where $\pi_2(P\circ A)$ defines the same $\Fq$ endomorphism of $\Fqm$ as $P\circ A$.
Therefore, in this instance, messages are easily encoded. The coefficients in the Bézout relation are already given above. As a result, the lifting of the Chinese Remainder Theorem is also easily computable. \\
It allows us to compute a generating matrix of this code:\\
Write $\displaystyle A(X) =\sum_{i=0}^{\alpha} a_i\cdot X^{q^i}$. A generating matrix of $\mathcal{C}_{F,k,A}$ is composed of two blocks.
Suppose $k>l.$ The first block is composed of $k$ rows of size $l$.
For every $1 \leqslant i\leqslant l, $ the $i$-th row of the first block is $(0,\dots,0,a_0^{q^{i-1}}, a_{l-1}^{q^{i-1}}, \dots, a_{l-i}^{q^{i-1}}), $ with $i-1$ zeros in the beginning. For every $l+1 \leqslant i\leqslant k, $ the $i$-th row of the first block is $\mathbf{ 0}_{l}$. Suppose also $\alpha<m$.
The second block is composed of $k$ rows of size $m$. Each row $i$ contains the coefficients of the polynomial $A^{q^i}$ divided by $X^{q^m}-X.$ The division of a $q$-polynomial $P$ by $X^{q^m}-X$ can be computed by replacing $X^{q^m}$ by $X$ in the expression of $P$. Therefore, for every $1 \leqslant i\leqslant m-\alpha, $ the $i$-th row of the second block is $(0, \dots, 0, a_0^{q^{i-1}}, \dots, a_\alpha^{q^{i-1}}, 0 , \dots, 0)$,  with $i-1$ zeros in the beginning, and $m-\alpha-i$ zeros in the end.  For every $ m-\alpha+1 \leqslant i\leqslant k, $ the $i$-th row of the second block is $(a_{m-i}^{q^{i-1}}, \dots, a_{\alpha}^{q^{i-1}},0 \dots, 0, a_0^{q^{i-1}}, \dots, a_{m-i-1}^{q^{i-1}} )$, with $m-\alpha-1$ zeros.

\subsubsection{Link with simple codes}
In this subsection, we will show that when the moduli polynomials are with coefficients in $\Fq$, the associated $q$CRT codes are simple codes.
\begin{definition}(\cite{simple})
      Let $C$ be an $[n,k]_{q^m}$ code. $C$ is said to be $(n,k,t)$-simple when it has a parity-check matrix $H$ of the form $$ H = \begin{bmatrix}
                  I_{n-k} & \begin{matrix}
                                  0_{t\times k} \\[3pt]
                                  A
                            \end{matrix}
            \end{bmatrix},$$ where $A\in \mathcal{M}_{ (n-k-t)\times k}(\Fqm)$.
\end{definition}

\begin{proposition}
      Let $F = (f_1, \cdots, f_s) \in (\mathbb{F}_q \langle X^q \rangle)^s$ with $f_i$
      of $q$-degree $d_i$ for all $i \in \{ 1, \cdots, s \}.$  Denote $n = \underset{i = 1}{\overset{s}{\sum}} d_i$. Let $k < n$ be an
      integer. Let $A \in \mathbb{F}_{q^m}\langle X^q \rangle$ be such that $\deg_q(A)+k<n.$ Denote $\alpha = \deg_q(A)$. \\
      The $q$CRT code $\mathcal{C}_{F, k, A}$ is a simple code.
\end{proposition}
\begin{proof}
      Consider the map
      \begin{equation*}
            L_n : \left\{ \begin{array}{rcl} \QP_{<n}             & \to     & \Fqm^n                \\
                         \sum_{i=0}^{n-1}p_iX^{q^i} & \mapsto & (p_0, \dots, p_{n-1})
            \end{array} , \right. \end{equation*}
      associating a $q$-polynomial to its vector coefficient.
      Denote $\mathcal{A} =\{P \circ A,~ P \in \QP_{<k} \}$. Consider $\mathcal{A}_0 = L_n(\mathcal{A})$.
      A generator matrix of $\mathcal{A}_0$ is given by $$ G = \Big(M~|~ \  {0} _{k\times(n-(k+\alpha))}\Big),$$ where $M\in \mathbb{F}_{q^m}^{k\times (k+\alpha)}$.

      We have $(L_n\circ  \Phi)(\mathcal{C})=\mathcal{A}_0$, where $\Phi$ denotes the lifting of the chinese remainder theorem,  and $L_n\circ  \Phi$ is an isometry for the rank metric. \\
\end{proof}
\color{black}

\subsection{Parity Check Matrices for $q$CRT Codes}
\label{sec:parity}

This section presents the construction of parity check matrices for $q$CRT codes and their duals. We first treat the simple case of two coprime $q$-polynomials before generalizing to arbitrary configurations.

\subsubsection{Two-Blocks Case}
\label{example}

Let $F = (f_1, f_2)$ be coprime $q$-polynomials with $\deg_q(f_i) = d_i$, and $A \in \mathbb{F}_{q^m}\langle X^q \rangle$ with $\deg_q(A) = \alpha$. For $k < d_1 + d_2 - \alpha$, consider the code $\mathcal{C}_{F,A,k} = \Psi_A(\QP_k)$.

\begin{lemma}
      A word $(p_1,p_2) \in \QP_{d_1} \times \QP_{d_2}$ belongs to $\mathcal{C}_{F,A,k}$ iff $\mathcal{L}(p_1,p_2) \in V_A$, where:
      \[
            \mathcal{L} :
            \begin{cases}
                  \QP_{d_1} \times \QP_{d_2} \to \QP_{d_1+d_2} \\
                  (p_1,p_2) \mapsto \pi_{1,2}(p_2 \circ S_1 \circ f_1 + p_1 \circ S_2 \circ f_2)
            \end{cases}
      \]
      and $V_A = \mathcal{M}_A(\QP_k)$ for $\mathcal{M}_A(F) = F \circ A$.
\end{lemma}

\begin{proposition}
      \label{prop:parity2}
      Let $K_A$ satisfying $V_A = \ker(K_A)$. Then $(p_1,p_2) \in \mathcal{C}_{F,A,k}$ iff $K_A \circ \mathcal{L}(p_1,p_2) = 0$. Any matrix representation of $K_A \circ \mathcal{L}$ yields a parity check matrix.
\end{proposition}

\begin{proof}
      Immediate from $\mathrm{Im}(\mathcal{M}_A|_{\QP_k}) = V_A = \ker(K_A)$.
\end{proof}

\subsubsection{General Case}
\label{general}

We now extend the construction to $s \geq 2$ blocks. Let $F = (f_1,...,f_s)$ with $\deg_q(f_i) = d_i$, $n = \sum_{i=1}^s d_i$, and $k < n - \alpha$. Define:

\begin{itemize}
      \item $h = \bigvee_{i=1}^s f_i$ (least common left multiple)
      \item $h_j = \bigvee_{i \neq j} f_i$ with $h_j \wedge_r f_j = X$
      \item $S_j \equiv h_j^{-1} \mod f_j$ (CRT inverses)
\end{itemize}

The generalized map becomes:
\[
      \Phi : \prod_{i=1}^s \QP_{d_i} \to \QP_{n}, \quad
      (P_1,...,P_s) \mapsto \pi\left( \sum_{i=1}^s P_i \circ S_i \circ h_i \right)
\]
where $\pi$ is the remainder under right division by $h$.

\begin{theorem}
      $\Phi$ is a left inverse of $\Pi$ (a retraction of the injection $\Pi$).
\end{theorem}

\begin{proof}
      For $P \in \QP_{k+\alpha}$, let $Q = \sum_{i=1}^s \pi_i(P) \circ S_i \circ h_i$. Since $h_j$ is a left multiple of $f_i$ for $j \neq i$, we have $\pi_i(Q) = \pi_i(P)$. The result follows from the injectivity of $\Pi$.
\end{proof}

\begin{proposition}
      \label{prop:parity}
      A word $\mathbf{P} = (P_1,...,P_s)$ is in $\mathcal{C}_{F,k,A}$ iff $K_A \circ \Phi(\mathbf{P}) = 0$, where $K_A$ is as in Proposition~\ref{prop:parity2}.
\end{proposition}

\begin{proof}
      Direct consequence of $\Phi(\mathbf{P}) \in V_A \Leftrightarrow \mathbf{P} \in \mathcal{C}_{F,k,A}$.
\end{proof}

\noindent As in the two-blocks case, each matrix expressing $K_A \circ \Phi$ gives a parity check matrix of  $\mathcal{C}_{F,k,A}$.

\section{A special case and a first decoding algorithm} \label{simpdecod}

In this section, we build $q$CRT codes with moduli linearized polynomials with coefficients in $\mathbb{F}_q$, so that the
map $\Pi$ of the Chinese Remainder Theorem and its lifting do not change the
support of the involved polynomials. Using this property, we design a probabilistic algorithm with a bound on its probability of failure, under the assumption that errors are uniformly distributed. \color{black}\\
The idea of this algorithm is to first determine the lifted
error support by computing a basis of this space, and then to compute the message if possible. The
support of the lifted is obtained using the Chinese remainder lifting. \\

The algorithm will be based on the fact that we can recover part of the lifted error only by considering the terms of higher degree in the lifted corrupted codeword. This property also appears in \cite{yu_polynomial_2012}. The strategy of the algorithm is the same as in \cite{yu_polynomial_2012}: from the information given by the upper part of the lifted error, we can compute the lower part of the lifted error, and then recover the message associated to the corrupted codeword.
\color{black}

\noindent Let $F = (f_1, \cdots, f_s) \in (\mathbb{F}_q \langle X^q \rangle)^s$ with $f_i$
of $q$-degree $d_i$ for all $i \in \{ 1, \cdots, s \}.$ \\
We denote, for all $i \in \{ 1, \cdots, s\},$ $b_i = \underset{j \neq
            i}{\vee_l} f_j$.

\noindent Suppose that $F$ satisfies the conditions to apply the Chinese Remainder Theorem \ref{CRT2}: we suppose that for all $i \in \{ 1, \cdots, s\},$ $f_i$ is coprime with $b_i$. Let $S_{1, i}$ and $S_{2, i}$ such that $S_{1, i} \circ b_i +
      S_{2, i} \circ f_i = X$ (i.e. $S_{1, i}$ is the inverse of $b_i$ mod $f_i$, and   $S_{2, i}$ is the inverse of $f_i$ mod $b_i$).\\
Denote $n = \underset{i = 1}{\overset{s}{\sum}} d_i$. Let $k < n$ be an
integer. Let $A \in \mathbb{F}_{q^m}\langle X^q \rangle$ be such that $\deg_q(A)+k<n.$ We denote $\alpha = \deg_q(A)$. \\

\noindent Let $\mathcal{C}=\mathcal{C}_{F, k,A}$ be the $q$CRT code associated
to $F$, $k$ and $A$.
We denote
$$\Phi : \left\{\begin{array}{l}
            \mathbb{F}_{q^m} \langle X^q \rangle_{d_1} \times \cdots \times
            \mathbb{F}_{q^m} \langle X^q \rangle_{d_s} \longrightarrow \mathbb{F}_{q^m}
            \langle X^q \rangle_{n} \\
            (p_1, \cdots, p_s) \longmapsto p
      \end{array}\right.$$ the lifting of the Chinese Remainder Theorem (since $\Pi$ is injective on $\mathbb{F}_{q^m} \langle X^q \rangle_{k+\alpha}$, it is the left inverse over its image) i.e. for all $p \in \mathbb{F}_{q^m} \langle X^q \rangle_{n},~ \Phi \circ \Pi(p) = p. $\\
Let $P \in \mathbb{F}_{q^m} \langle X^q \rangle_{k}$.
We denote $c= \Pi (P \circ A) = (\pi_1 (P \circ A), \cdots,
      \pi_s (P\circ A))$. Consider $y=c+e$, where
$e \in \QP_n$ is an error.
As $\Phi $ is $\mathbb{F}_{q^m}$-linear, \[ \Phi (y)
      = \Phi (c+e) = \Phi (c) + \Phi
      (e) = P \circ A + \Phi(e). \]
As in \cite{yu_polynomial_2012}, we will write the lifted error as the sum of an upper part and a lower part:\\
\color{black}
We denote $\Phi
      (e) = \underline{E} + \overline{E}$ where $\underline{E} \in
      \mathbb{F}_{q^m} \langle X^q \rangle_{k+\alpha}$ and $\overline{E} =
      \underset{}{\overset{n - 1}{\underset{i =k+ \alpha}{\sum}} e_i \cdot X^{q^i}}$.
We then have
\begin{equation} \label{remontee}
      \Phi(y) =  P \circ A + \underline{E} + \overline{E},
\end{equation}
and as $\deg_q( P \circ A + \underline{E})<k+ \alpha,$ we can deduce $\overline{E}$ from the monomials of $\Phi(y)$ of $q$degree superior or equal to $k+\alpha$. \\

\noindent Denote $\mathcal{E}:=  supp( E )$ the $\mathbb{F}_q$-vector subspace of
$\mathbb{F}_{q^m}$ generated by the coefficients of $E$ and in the same way,
we denote $\underline{\mathcal{E}} := \supp( \underline{E} )$
and $\overline{\mathcal{E}}: = \supp( \overline{E} )$ those
associated to $\underline{E}$ and $\overline{E}$ respectively.

\subsection{Computing the error support } \label{section_Computing_the_error_support}

This subsection is devoted to the computation of a basis of the support of
the lifting of the error $\mathcal{E}$. The situation is very similar to the one with augmented Gabidulin codes (see \cite{aragon_minrank_2025-1}). The case of interest is when $\dim \left( \overline{\mathcal{E}}
      \right) = \dim (\mathcal{E})$. We can bound the probability that
$\overline{\mathcal{E}} =\mathcal{E}$,  supposing that the error follows a uniform distribution. \color{black}

\begin{proposition} \label{proba}
      Let $r \in \{1, \dots, \min(m,n-\alpha-k)\}.$ We suppose that $e$ follows a uniform distribution on $\Fqm^n.$
      Knowing that $w_r(E)=r$, we have $\overline{\mathcal{E}} =\mathcal{E}$ with probability \[ q^{r(k+\alpha)} \prod_{i=0}^{r-1} \frac{(q^{n-(k+\alpha)}-q^i)}{(q^n -q^i)}.\]
\end{proposition}
\begin{remark}
      If $w_r(E)=0,$ then $\overline{\mathcal{E}} =\mathcal{E}$, and if $w_r(E)>n- \alpha-k,$ then $\overline{\mathcal{E}} \subsetneq \mathcal{E}.$
\end{remark}

\begin{proof}Since $\overline{\mathcal{E}}  \subset \mathcal{E},$ $\overline{\mathcal{E}}$ and $\mathcal{E}$ are equal if and only if they have the same dimension.
      As $e$ follows a uniform distribution on the set of vectors of $\Fqm^n$, $E$ also follows a uniform distribution on this set. Hence, \[\mathbb{P}(w_r(\overline{E})=r|w_r(E)=r) = \frac{\text{Card}( \{v \in  \Fqm^n,~ w_r(v_2) = r \text{ and } w_r(v) = r\})}{\text{Card}( \{v \in  \Fqm^n,~ w_r(v) = r\}},\]
      where for every $v \in \mathbb{F}_{q^m}^n,$ $v_2$ denotes the vector composed of the last $n-k-\alpha$ coordinates of $v.$

      \noindent On the one hand, the set $\{v \in  \Fqm^n,~ w_r(v) = r\}$ has the same cardinality as the set of matrices of size $m\times n$ with coefficients in $\Fq$ of rank $r.$ Therefore (see \cite{pless_handbook_nodate}), \[\text{Card}( \{v \in  \Fqm^n, w_r(v) = r\} = \prod_{i=0}^{r-1} \frac{(q^n-q^i)(q^m-q^i)}{(q^r- q^i)}.\]
      On the other hand, {\small \[\text{Card}( \{v \in  \Fqm^n,~ w_r(v_2) = r \text{ and } w_r(v) = r\}) = q^{r(k+\alpha)} \prod_{i=0}^{r-1} \frac{(q^{n-(k+\alpha)}-q^i)(q^m-q^i)}{(q^r- q^i)}.\] }

      \noindent Indeed, building a vector in this set is equivalent to build a vector in $\Fqm^n$ of two blocks of size $k+\alpha$ and $n-(k+\alpha)$, where the left block has rank weight $r$ and all the coordinates of the right block belong to the the support of the first block. If we build such a vector, the number of possibilities for the left block is the number of matrices of size $m \times (n-(k+\alpha))$ with coefficients in $\Fq$ of rank $r,$ that is \[\prod_{i=0}^{r-1} \frac{(q^{n-(k+\alpha)
                        }-q^i)(q^m-q^i)}{(q^r- q^i)}. \] Once the left block is chosen, as its support has dimension $r,$ the right block is built by choosing uniformly $k+\alpha$ coordinates in a space of cardinality $q^r$, and therefore the number of possibilities is $q^{r(k+\alpha)}.$

      \noindent Consequently, \[ \mathbb{P}(w_r(\overline{E})=r|w_r(E)=r) = q^{r(k+\alpha)} \prod_{i=0}^{r-1} \frac{(q^{n-(k+\alpha)}-q^i)}{(q^n -q^i)}. \]
\end{proof}

\noindent Depending on the choice of the parameters, this probability can be close to one for many values of $r$, where $r$ is the rank weight of the lifting of the error. This means that the second block can allow to recover all the support of the lifting of the error with high probability for many values of $r$. Note that this probability does not depend on the value of $m$.  \\

\noindent Since for every $i \in \{1, \dots, n\},$ the coefficients of the polynomial $f_i$ are elements of $\Fq,$ the support of the lifted error is contained within the support of the error (see Proposition \ref{inclusion support}), $\operatorname{supp}(\Phi(e)) \subseteq \operatorname{supp}(e)$. This fact guarantees that the dimension of the support of the lifting of the error will be smaller than the one of the error. Consequently, if the error's support has dimension $r$, the probability that $\dim \left( \overline{\mathcal{E}}
      \right) = \dim (\mathcal{E})$ is bounded by the value given in Proposition \ref{proba}.

\noindent To show this property, we need the following lemmas.

\begin{lemma} \label{compo fq}
      Let $g \in \QP$, and $f\in \Fq\langle X^q \rangle.$ Suppose $\supp(g) \subset S$,  where $S$ is  $\Fq$ vector subspace of $\Fqm$,  then \[\supp(g \circ f) \subset S.\]
      In particular,  \[\supp(g \circ f) \subset \supp(g).\]
\end{lemma}

\begin{proof}
      Denote $\displaystyle f = \sum_{i=0}^{d_f} f_i \cdot X^{q^i}$ and $\displaystyle g = \sum_{i=0}^{d_g} g_i \cdot X^{q^i}. $
      We have \[(g \circ f) = \underset{l = 0}{\overset{d_f + d_g}{\sum}} h_l \cdot X^{q^l},\text{ with }h_l = \underset{i + j = l}{\sum} g_i \cdot
            f_j^{q^i}.\] Since for all $ i \in \{ 1, \cdots, d_f\}, $ $f_i \in \Fq$, and for all  $ i \in \{ 1, \cdots, d_g\}, $ $g_i \in S$, for all $l \in \{0, \cdots, d_f + d_g\}$, we have $h_l \in S.$ Therefore $\supp(g \circ f) \subset S$.
\end{proof}

\begin{lemma}\label{reste}
      If $g \in \QP$, and $f\in \Fq\langle X^q \rangle,$ then the remainder $\pi(g)$ of $g$ in the right division by $f$ satisfies: \[ \supp(\pi(g)) \subset \supp(g).\]
\end{lemma}
\begin{proof}
      We prove the result by induction on the steps of the algorithm of the right division of $g$ by $f.$ \\
      We denote, for $i\geq 1,$ $R_i$ the value of $R$ at the step $i$ in the algorithm \ref{rqorem}, and $R_0=g.$ We prove that for every $i \geq 0,~ \supp(R_i) \subset \supp(g)$.
      First, since $R_0=g$, the result is clear for $i = 0.$ \\
      Suppose there exists $i \geq 0$ such that $\supp(R_i) \subset \supp (g). $
      We have \[R_{i+1} = R_i-(\LC(R_i)\cdot  \LC(f)^{q^{d_f-d_{R_i}}}X^{q^{d_{R_i}-d_f}})\circ f. \]
      Since $\LC(f) \in \Fq,$ and by induction hypothesis,  $\LC(f)^{q^{d_f-d_{R_i}}}X^{q^{d_{R_i}-d_f}}$ is a polynomial with coefficients in $\supp(g).$
      By Lemma \ref{compo fq}, \[\supp((\LC(R_i)\cdot  \LC(f)^{q^{d_{R_i}-d_f}}X^{q^{d_f-d_{R_i}}})\circ f) \subset \supp(g).\] Therefore, $\supp(R_{i+1}) \subset \supp(g).$ \\
      If the last step is $N$, by taking $i=N,$ we obtain the result.
\end{proof}

\begin{proposition} \label{inclusion support}
      We have \[ \supp(\Phi(e)) \subset \supp(e).\]
\end{proposition}
\begin{proof}
      Denote for every $g \in \QP$, $\pi_{1,i}(g)$ the remainder of $g$ in the right division by $h_i.$ For all $i \in \{1, \cdots, s\},$ denote $b_i = \underset{j \neq i}{\vee_l} f_j$, and  $S_{1,i}$ and $S_{2,i}$  coefficients such that $S_{1,i} \circ  b_i + S_{2,i} \circ f_i = X. $
      We have, by Theorem \ref{CRT2},\[ \Phi(e) = \pi_{1,s}(\sum_{i=1}^s \pi_i(e) \circ S_{1,i} \circ b_i).\]
      We deduce from the fact that for all $i \in \{1, \cdots, s \}, ~ f_i \in \Fq\langle X^q \rangle$ that $b_i \in \Fq\langle X^q \rangle,$ and that $S_{1,i}$ and $S_{2,i}\in \Fq\langle X^q \rangle$. By Lemma \ref{compo fq} and \ref{reste}, for every $ i \in \{ 1, \cdots,s\}$, we have $\supp(  \pi_i(e) \circ S_{1,i} \circ b_i) \subset \supp(e).$ Therefore  applying Lemma \ref{reste} again, $\supp(\Phi(e)) \subset \supp(e).$
\end{proof}

\subsection{Computing the error}

Consider the set $\mathcal{M}_A(\QP_{k}) = \{P \circ A,~ P \in \QP_{k} \} $. Seeing the $q$ polynomials of $q$-degree smaller than $k+\alpha$ as vectors of length $k+\alpha$ with coefficients in $\Fqm,$ we can see the set $\mathcal{M}_A(\QP_{k})$ as a vector subspace of $\Fqm^{k+\alpha}$ of dimension $k$, and thus, as a code of length $k+\alpha$ and dimension $k$ over $\Fqm$. Consider $H \in \mathcal{M}_{\alpha \times (k+\alpha) }(\Fqm)$ a parity check matrix of $\mathcal{M}_A(\QP_{k})$.

\noindent We then have, for every $P \in \QP_{k}, $ \[ H(P\circ A) =0, \] and therefore,  \[ H(P\circ A + \underline{E}) = H \underline{E}.  \]
Once we know $\supp(E),$ to compute $\underline{E}$, one can compute $s:= H(P\circ A + \underline{E})$, and solve the system $s= H(\underline{E}).$ (see \cite{gaborit_low_2013}). \\
This system is a system with $r(k+\alpha)$ unknowns and $m\alpha$ equations over $\Fq,$ which we can solve if $r<\frac{m\alpha}{k+\alpha}.$

\noindent We are then able to recover $P\circ A$ by subtracting $E$ to $\Phi(y).$ We then deduce $P$ by dividing on the right $P\circ A$ by $A$. \\

\subsection{Decoding algorithm}

In this subsection, we sum up the decoding algorithm.
\noindent We keep the notations from the last subsection.
Recall that $\mathcal{C}= C_{F,k,A}, $  and that $H$ is a parity check matrix of $\mathcal{M}_A(\QP_{k})$.

\begin{algorithm}[htbp]
      \caption{Decoding algorithm in the case where the moduli polynomials have coefficients in $\Fq$.}
      \begin{algorithmic}[1]
            \REQUIRE $ y =  c +  e$, where $c \in \mathcal{C},$ and $w_r(e) \leqslant \min( \frac{m\alpha}{k+\alpha },n-k-\alpha).$
            \ENSURE  $P \in \QP_{k}$ such that $d_r(y,\Psi_A (P))\leqslant w_r(e),$ or failure.
            \STATE Compute $Y := \Phi(y) \in \QP$.
            \STATE From $Y$, deduce $\overline{E}$, by computing the sum of the monomials of $Y$ of $q$-degree superior or equal to $k+\alpha$, and compute $\supp(\overline{E}).$
            \STATE  Compute $s:= H(Y-\overline{E}),$ and solve the system $H\underline{E}=s$, using $\supp(\overline{E})$, to find $\underline{E}$.
            \STATE Deduce $P\circ A,$ by computing $Y-\underline{E}-\overline{E}$.
            \STATE Compute the right division of $P\circ A$ by $A.$
      \end{algorithmic}
\end{algorithm}

This algorithm exhibits a certain decoding failure rate. Specifically, it fails when $\supp(\overline{E}) \subsetneq \supp(E)$, as it cannot recover $\supp(E)$ under this condition. Supposing that the error is uniformly drawn, \color{black}the probability that $\supp(\overline{E}) = \supp(E)$ is detailed in Proposition \ref{proba}.

When $\supp(\overline{E}) \subsetneq \supp(E)$, the codeword associated with the algorithm's output is not necessarily the closest codeword to the input. To determine if the algorithm has failed, one can verify whether the codeword associated with the output has a rank distance of at most $r$ from the input, where $r = w_r(e)$.

It is noteworthy that if $\min\left(\frac{m\alpha}{k+\alpha}, n-k-\alpha\right)$ exceeds $\frac{n-k}{2}$, the algorithm can decode the input even when the rank weight of the error exceed the unique decoding radius. In such cases, however, the codeword associated with the algorithm's output is not guaranteed to be the closest codeword to the input.

\subsection{Examples of probabilities of success}

In this section, we present graphs illustrating the success probability of the decoding algorithm, under the assumption that the error follows a uniform distribution.\color{black} The blue curve represents the probability of obtaining a basis of the support of the lifted error using its second block, as a function of the rank weight of the lifted error. The red line indicates the bound given by the linear system, while the black line corresponds to the unique decoding radius.

\noindent
Using the same notation as above, the parameters used to generate the following graph are $n=200$, $k=50$, $\alpha=50$, $q=5$, and $m=80$. \\

\begin{figure} [H]
      \centering
      \includegraphics[scale = 0.4]{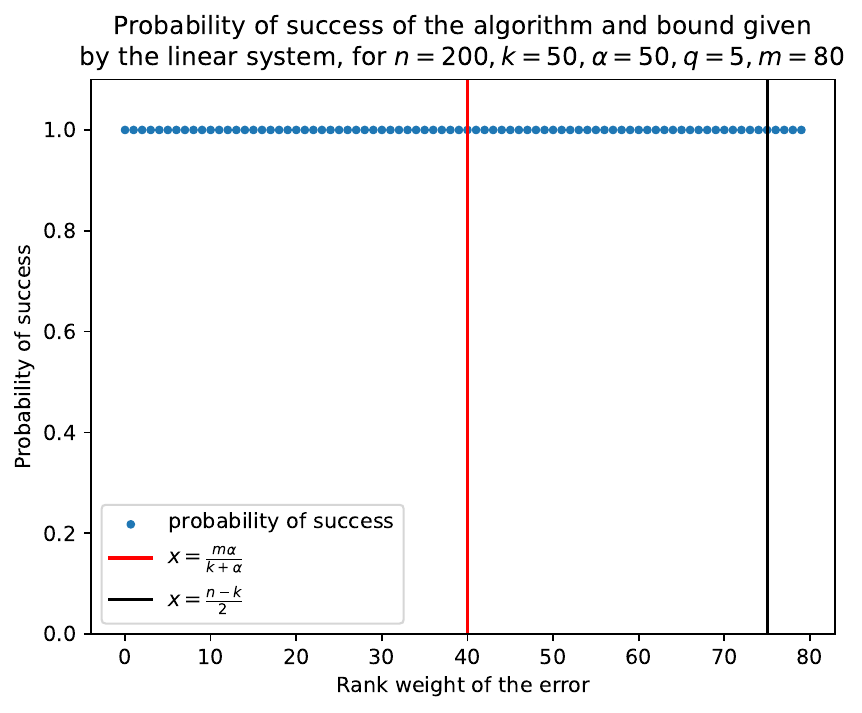}\\
      \caption{We can observe that the probability of success is close to one for all rank weights that the algorithm can decode, i.e. all the rank weights smaller than the bound given by the linear system.}
\end{figure}

\noindent We can modify the relative positions of the bounds by changing the parameters of the code. For example, in the following graphs, we fix the code length, the dimension, the $q$-degree of the polynomial $A$, and the characteristic of the alphabet. We vary the value of $m$, which is the degree of the extension of the alphabet over $\Fq$. The parameters used to generate the graphs are: $n = 70$, $k = 15$, $\alpha = 14$, and $q = 2$.
\begin{figure}[H]
      \centering
      \includegraphics[scale = 0.4]{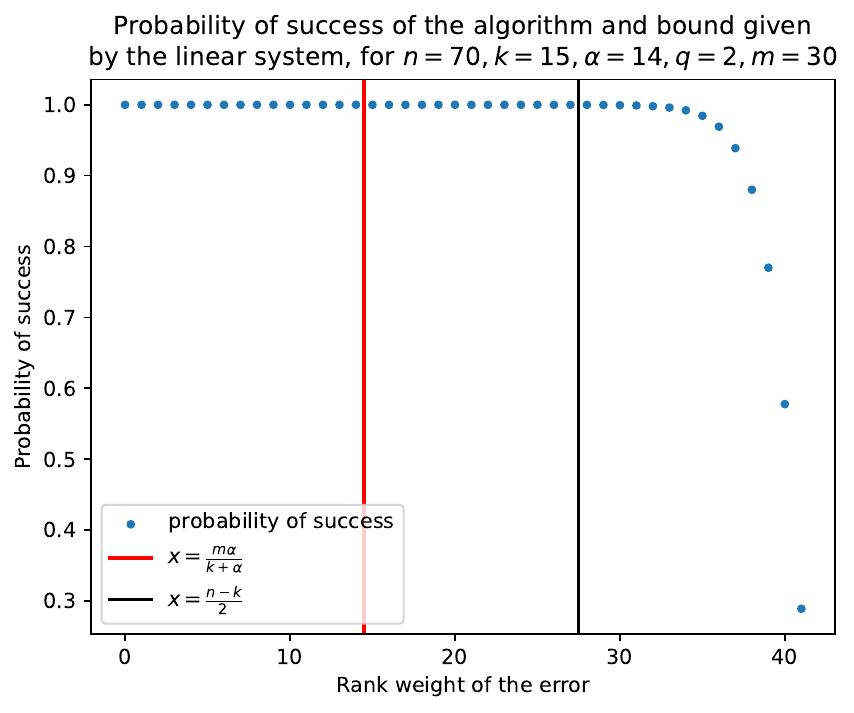}
      \caption{In this example, the bound given by the linear system is smaller than the unique decoding radius. The probability is close to one for all rank weights smaller than this bound. }
      \label{fig: avant dec unique}
\end{figure}

\begin{figure} [H]
      \centering
      \includegraphics[scale = 0.4]{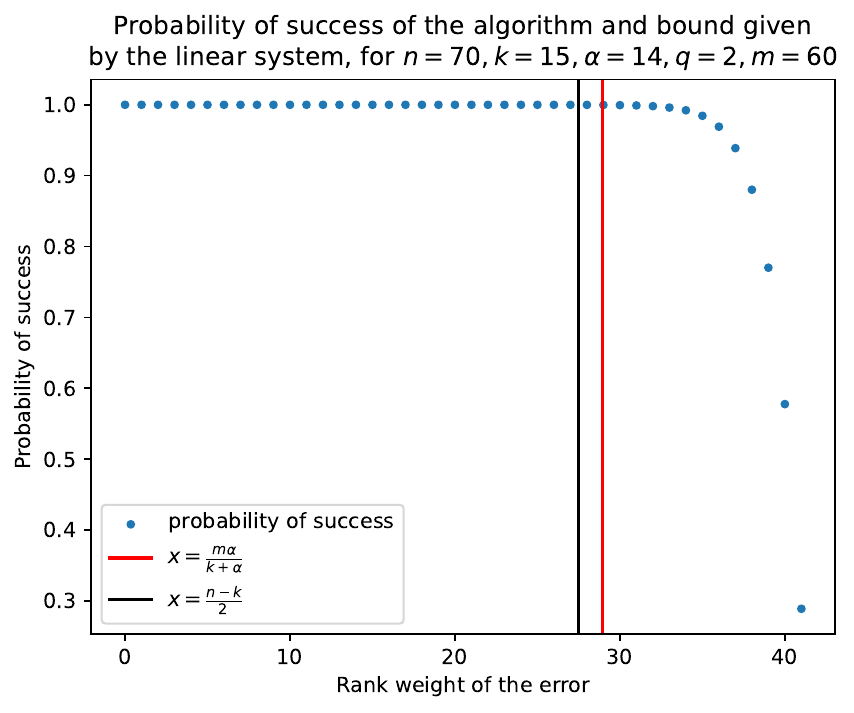}
      \caption{In this example, the bound given by the linear system is larger than the unique decoding radius. The probability is close to one for all rank weights smaller than this bound.}
      \label{fig: apres dec unique}
\end{figure}
\begin{figure}[H]
      \centering
      \includegraphics[scale = 0.4]{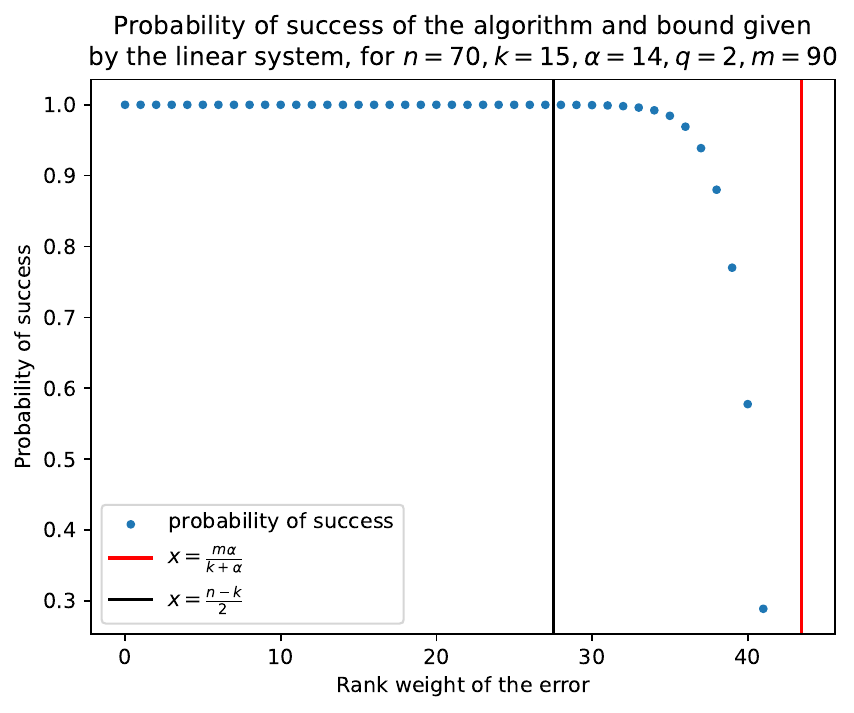}
      \caption{In this example, the bound given by the linear system is larger than the unique decoding radius. However, the probability drops for rank weights below this bound, indicating that the algorithm may fail before reaching the limit imposed by the linear system.}
      \label{fig: apres chute proba}
\end{figure}

Theses figures illustrate the fact that the value of $m$ does not affect the probability. However, it does impact the bound given to be able to solve the linear system to recover the lifted error. Depending on the value of $m$, this bound can be either smaller or larger than the unique decoding radius. This implies that the choice of $m$ can determine whether the algorithm is able to decode up to the unique decoding radius: \\
In figure \ref{fig: avant dec unique}, the algorithm cannot decode up to the unique decoding radius. In this case, the limit of the algorithm is the bound given by the linear system to recover the lifted error, because the probability of recovering the support of the lifted error with its second block is close to one for all rank weights smaller than this bound.  Increasing the value of $m$ may allow the algorithm to reach the unique decoding radius, as shown in figure \ref{fig: apres dec unique}. In this case again, the limit of the algorithm is the bound given by the linear system to compute the lifted error. Moreover, selecting an appropriate $m$ can ensure that the algorithm only attempts to decode errors for which the success probability is high. Indeed, if the value of $m$ is too large, the limit of the algorithm can becomes the probability: in figure \ref{fig: apres chute proba}, the probability of recovering the full support of the lifted error drops for rank weights below the bound given by the linear system to compute the lifted error.

\section{Decoding a wider class of $q$-CRT codes} \label{widerdecod}

Here, we consider the case where $f_1, \cdots, f_s \in \mathbb{F}_{q^l}
      \langle X^q \rangle$ with $l$ small. We explain how the algorithm of the
previous section can be extended to this case. The algorithm is almost the
same. The main impact is on the bounds for the probabilities of decoding. We
use the same notations than the above section. We introduce a new notion. If
$\mathcal{A}$ and $\mathcal{B}$ are two $\mathbb{F}_q$-subvector spaces of
$\mathbb{F}_{q^m}$, we denote $\mathcal{A} \cdot \mathcal{B}= \langle a \cdot
      b | a \in \mathcal{A} \text{ and } b \in \mathcal{B} \rangle \nobracket$ the
``product of the two vector spaces''. If they have dimension respectively $l$
and $k$ and $\mathcal{A}= \langle a_1, \cdots, a_l \rangle$ and $\mathcal{B}=
      \langle b_1, \cdots, b_k \rangle$ then $\mathcal{A} \cdot \mathcal{B}$ has
dimension less than $l \cdot k$ and $\left\{ a_i \cdot b_j | \nobracket i \in
      \{ 1, \cdots, l \}\text{ and } k \in \{ 1, \cdots, k
      \} \right\}$ is a generating set of $\mathcal{A} \cdot \mathcal{B}$. \

In order to decode $\mathcal{C}_{F, A, k}$, we still proceed in two steps.
The first one consists in finding a basis of the lifted error support and the second
to recover the message given the support of the error.

Just as the previous section if $y = c + e$ with $c \in \mathcal{C}_{F, A,
            k}$ and $e \in \mathbb{F}_{q^m} \langle X^q \rangle_n$ with $w_r(e)
      \leqslant r$, we compute $\Phi (y) = \Phi (c) + \Phi (e) = P \circ A + E$ with
$P$ of $q$-degree strictly less than $k$ and we have $\Phi (y) = P \circ A +
      \underline{E} + \overline{E}$ just as in \ref{remontee}, $\deg_q \left(
      \underline{E} \right) < k+\alpha$ and $\underline{E} (X) = \underset{i =
            0}{\overset{k+\alpha - 1}{\sum}} e_i \cdot X^{q^i}$ and $\overline{E} (X) =
      \overset{n - 1}{\underset{i = k +\alpha}{\sum}} e_i \cdot X^{q^i}$. We denote $E (X) =
      \underline{E} (X) + \overline{E} (X)$, $\mathcal{E}= \tmop{supp} (E)$,
$\underline{\mathcal{E}} = \tmop{supp} \left( \underline{E} \right)$ and
$\overline{\mathcal{E}} = \tmop{supp} \left( \overline{E} \right)$. The main
difficulty arises from the fact that $\Phi$ does not preserve the support
anymore. It is to say that $\tmop{supp} (E)\nsubseteq \tmop{supp} (e)$.

\begin{proposition}
      We have $\tmop{supp} (E) \subset \tmop{supp} (e) \cdot \mathbb{F}_{q^l}$ and
      then $\dim (\tmop{supp} (E)) \leqslant r \cdot l$.
\end{proposition}

In fact, one can consider the preceding algorithm but taking $l \cdot r$
instead of $r$ for the bound of the probability to get  the whole support of
the error on the upper part, supposing that the error follows a uniform distribution.\color{black}

\begin{proposition}
      Let $r \in \left\{ 1, \cdots, \min \left( \left\lfloor \frac{m}{l}
            \right\rfloor, \left\lfloor \frac{n - \alpha - k}{l} \right\rfloor \right)
            \right\}$. We assume that $e$ follows a uniform distribution on
      $\mathbb{F}_{q^m}$. Knowing that $w_r (E) \leqslant r \cdot l$, we have
      $\overline{\mathcal{E}} =\mathcal{E}$ with probability bounded by $q^{l \cdot r \cdot
                        (k+ \alpha)} \underset{i = 0}{\overset{r \cdot l - 1}{\prod}} \frac{q^{n -(k+ \alpha)} -
                  q^i}{q^n - q^i}$.
\end{proposition}

If we have computed a basis of $\tmop{supp} (E)$, solving the linear system to compute the coefficients of $E$ like in the previous decoding algorithm requires that $r\leqslant \frac{m\alpha}{l(k+\alpha)}$.

\begin{remark}
      Here we do not take advantage that we know that $\tmop{supp}(E) \subset \tmop{supp}(e) \cdot \GF{q^l}$. One can try to use the approach used for syndrome decoding of LRPC codes (see \cite{gaborit_low_2013}) in order to improve the probability. But it will not improve the bound imposed by the linear system to recover the coefficients of $E$.
\end{remark}

\subsection{Examples of probabilities of success}

In this section, we present graphs illustrating the success probability of the decoding algorithm as a function of the rank weight of the lifted error, for different values of $l$, where $l$ is the degree of the extension in which the coefficients of the moduli polynomials reside. We suppose that the error follows a uniform distribution in the set of vectors with coefficients in $\Fqm$ and length $n$.\color{black}

\noindent Keeping the same notation as above, the parameters chosen for these examples are $n = 200$, $k = 30$, $\alpha = 10$, $q = 2$, and $m = 100$.

\begin{figure}[H]
      \centering
      \includegraphics[scale = 0.4]{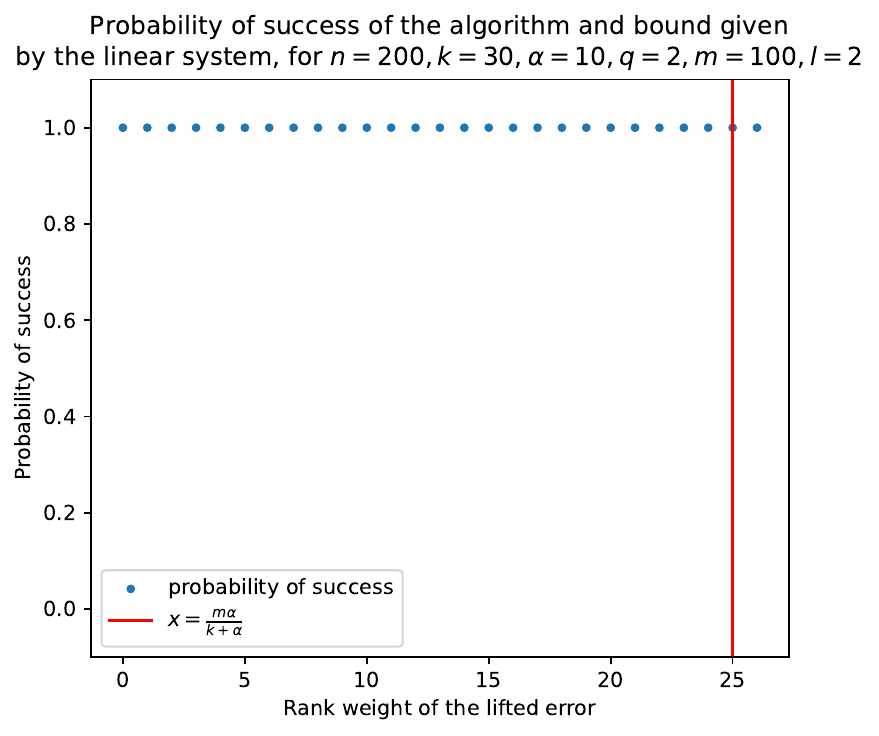}
      \caption{For small values of $l$, the probability remains close to 1 for all rank weights smaller than the bound of the linear system. }
      \label{fig: l2}
\end{figure}

\begin{figure}[H]
      \centering
      \includegraphics[scale = 0.4]{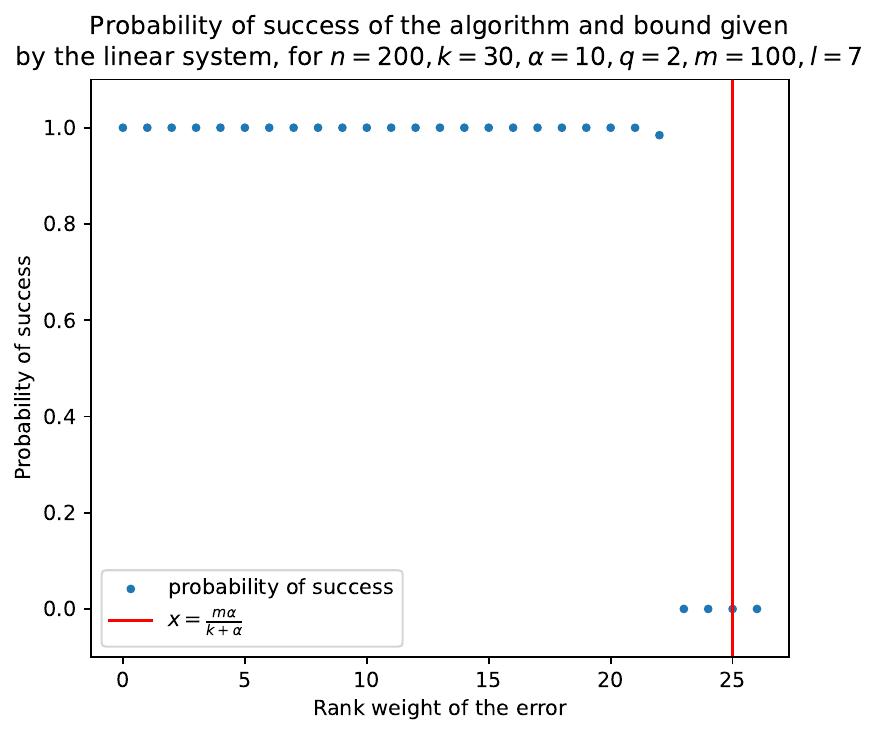}
      \caption{The probability drops before the bound given by the linear system.}
      \label{fig: l7}
\end{figure}
\begin{figure}[H]
      \centering
      \includegraphics[scale = 0.4]{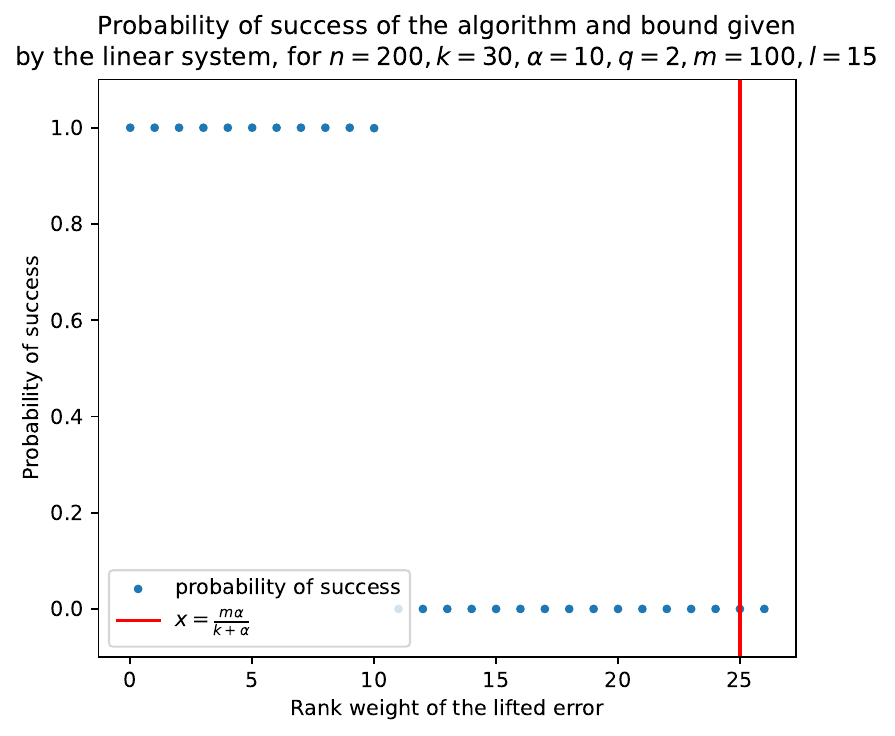}
      \caption{When $l$ increases, the success probability drops for smaller rank weights.}
      \label{fig: l15}
\end{figure}
If $l$ is small enough, the probability of recovering the full support of the lifted error remains close to 1 for all rank weights that the algorithm is able to decode (see Figure \ref{fig: l2}). When $l$ becomes larger, the probability drops before the bound given by the linear system, and the limiting factor of the algorithm becomes the probability of recovering the full support (see Figure \ref{fig: l7}). As $l$ increases, the success probability drops for smaller rank weights: the larger $l$ is, the earlier the probability begins to decline.

\section{Conclusions and perspectives}

In this paper, we presented a novel family of codes designed for both rank and sum-rank metrics, offering a rich parameter space for constructing codes with tailored properties. Our work introduces linearized Chinese Remainder Theorem ($q$CRT) codes, which not only expand the landscape of available codes but also provide insights into the structure and behavior of these important algebraic objects.

We have made significant progress in developing efficient decoding algorithms for our new code family. Specifically, we provided: \\
A polynomial-time decoding algorithm with a controlled error rate for a particular subclass of linearized CRT codes.\\
An extension of this algorithm to a broader class, demonstrating its versatility and applicability.

Our future research directions are twofold:

Improved Decoding: We aim to develop a decoding algorithm that maintains polynomial running time but eliminates the failure rate observed in our current algorithm. This would enable more robust and reliable communication and storage systems based on linearized CRT codes.\\
Applications: We plan to explore various applications of these new codes, including:\\
Local decoding in sum-rank metric for improved error correction capabilities in distributed storage systems for instance.\\
Probabilistic checkable proofs, leveraging the unique properties of our codes to enhance verification efficiency in cryptographic protocols.\\
Distributed storage and communication, where our codes' adaptability can lead to more resilient and efficient networks.\\
Cryptography, where the rich parameter space of linearized CRT codes may enable novel constructions or improvements to existing cryptosystems.\\
By pursuing these avenues, we expect to further unlock the potential of rank and sum-rank-metric codes, driving advancements in both theoretical understanding and practical applications.

\tmstrong{Acknowledgements:} The authors would like to thank Alain Couvreur for pointing out the connection with simple codes, and the anonymous referees for their valuable suggestions.

\bibliographystyle{plain}
\bibliography{QCRT}

\end{document}